\documentclass[12pt]{amsart}
\usepackage{latexsym,amssymb,amsmath,graphicx,hyperref,subfig,caption,verbatim,ulem}
\usepackage{tikz}
\usetikzlibrary{shapes}
\usetikzlibrary{arrows}
\usepackage{fp}
\usepackage{rotating}
\usepackage{longtable}
\usepackage{mathtools}
\usepackage{centernot}

\numberwithin{equation}{section}
\newtheorem{theorem}{Theorem}[section]
\newtheorem{lemma}[theorem]{Lemma}

\newtheorem{corollary}[theorem]{Corollary}

\newtheorem{question}[theorem]{Question}

\theoremstyle{definition}
\newtheorem{definition}[theorem]{Definition} 
\newtheorem{remark}[theorem]{Remark}
\newtheorem{example}[theorem]{Example}

\usepackage{color} 
\definecolor{rossred}{rgb}{1.0,0.25,0.66}

\def\NZQ{\mathbb}               

\def\ZZ{{\NZQ Z}}

\def\opn#1#2{\def#1{\operatorname{#2}}} 
	\opn\chara{char} \opn\length{\ell} \opn\pd{pd} \opn\rk{rk}
	\opn\projdim{proj\,dim} \opn\injdim{inj\,dim} \opn\rank{rank}
	\opn\gcd{gcd}
    \opn\depth{depth} \opn\grade{grade} \opn\height{height}
	\opn\embdim{emb\,dim} \opn\codim{codim}
	\opn\Cl{Cl}
	
	\opn\Tr{Tr} \opn\bigrank{big\,rank}
	\opn\superheight{superheight}\opn\lcm{lcm}
	\opn\trdeg{tr\,deg}
	\opn\rdeg{rdeg}
	\opn\reg{reg} \opn\lreg{lreg} \opn\ini{in} \opn\lpd{lpd}
	\opn\size{size} \opn\sdepth{sdepth}
	\opn\link{link}\opn\fdepth{fdepth}\opn\lex{lex}
	\opn\tr{tr}
	\opn\type{type}
	\opn\gap{gap}
	\opn\arithdeg{arith-deg}
	\opn\revlex{revlex}
	\opn\div{div} \opn\Div{Div} \opn\cl{cl} \opn\Cl{Cl}
	\opn\Spec{Spec} \opn\Supp{Supp} \opn\supp{supp} \opn\Sing{Sing}
	\opn\Ass{Ass} \opn\Min{Min}\opn\Mon{Mon}
	\opn\Ann{Ann} \opn\Rad{Rad} \opn\Soc{Soc}
	\opn\Im{Im} \opn\Ker{Ker} \opn\Coker{Coker} \opn\Am{Am}
	\opn\Hom{Hom} \opn\Tor{Tor} \opn\Ext{Ext} \opn\End{End}
	\opn\Aut{Aut} \opn\id{id}
	
	\opn\nat{nat}
	\opn\pff{pf}
	\opn\Pf{Pf} \opn\GL{GL} \opn\SL{SL} \opn\mod{mod} \opn\ord{ord}
	\opn\Gin{Gin} \opn\Hilb{Hilb}\opn\sort{sort}
	\opn\PF{PF}\opn\Ap{Ap}
	\opn\mult{mult}
	\opn\bight{bight}
	\opn\div{div}
	\opn\Div{Div}
    \opn\Ind{Ind}
	\opn\aff{aff}
	\opn\relint{relint} \opn\st{st}
	\opn\lk{lk} \opn\cn{cn} \opn\core{core} \opn\vol{vol}  \opn\inp{inp} 
	\opn\nilpot{nilpot}
	\opn\link{link} \opn\star{star}\opn\lex{lex}\opn\set{set}
	\opn\width{wd}
	\opn\Fr{F}
	\opn\QF{QF}
	\opn\G{G}
	\opn\type{type}\opn\res{res}
	\opn\conv{conv}
	\opn\Int{Int}
	\opn\Deg{Deg}
	\opn\Sym{Sym}
	\opn\Con{Con}
	\opn\gr{gr}
	
	\def\pot#1#2{#1[\kern-0.28ex[#2]\kern-0.28ex]}

	\opn\dirlim{\underrightarrow{\lim}}
	\opn\inivlim{\underleftarrow{\lim}}
	\def\Implies{\ifmmode\Longrightarrow \else
		\unskip${}\Longrightarrow{}$\ignorespaces\fi}
	\def\implies{\ifmmode\Rightarrow \else
		\unskip${}\Rightarrow{}$\ignorespaces\fi}
	\def\iff{\ifmmode\Longleftrightarrow \else
		\unskip${}\Longleftrightarrow{}$\ignorespaces\fi}

	\let\:=\colon
\usetikzlibrary{intersections, calc}

\newcommand{\PIVAL}{3.14159265358979323846264338} 
\newcounter{i} 
\newcommand{\Circulant}[2] { 
	\begin{tikzpicture}
	\setcounter{i}{0}
	\whiledo{\value{i}<#1}{ 
		\FPmul\tempA{2}{\thei} 
		\FPdiv\tempB{\PIVAL}{#1} 
		\FPmul\tempC{\tempA}{\tempB} 
		\FPcos\varX{\tempC} 
		\FPsin\varY{\tempC} 
		\stepcounter{i} 
		\FPround\varX{\varX}{3}
		\FPround\varY{\varY}{3}
		\node (\thei) at (\varX,\varY)[place]{ }; 
		\foreach \x in {#2} { 
			\pgfmathparse{mod(\x+\thei,#1)} 
			\let\tempB\pgfmathresult
			\pgfmathparse{mod(\thei-\x,#1)} 
			\let\tempA\pgfmathresult
			\ifthenelse{\lengthtest{\tempA pt < 1 pt}}{\FPadd\tempA{\tempA}{#1}}{}
			\ifthenelse{\lengthtest{\tempB pt < 1 pt}}{\FPadd\tempB{\tempB}{#1}}{}
			\ifthenelse{\lengthtest{\tempA pt > \thei pt}}{}{\ifthenelse{\thei = \tempA}{}{\draw [] (\thei) to (\tempA)}};
			\ifthenelse{\lengthtest{\tempB pt > \thei pt}}{}{\ifthenelse{\thei = \tempB}{}{\draw [] (\thei) to (\tempB)}};
		}
	}
	\end{tikzpicture}
}

\begin{document}

\tikzstyle{place}=[draw,circle,minimum size=0.5mm,inner sep=1pt,outer sep=-1.1pt,fill=black]

 
\title{Levelable graphs}
\thanks{Version: \today}

\author[K. Bhaskara]{Kieran Bhaskara}
\address[K. Bhaskara]{Department of Mathematics and Statistics\\
McMaster University, Hamilton, ON, L8S 4L8, Canada}
\email{kieran.bhaskara@mcmaster.ca}

\author[M. Y. C. Chong]{Michael Y. C. Chong}
\address[M. Y. C. Chong]{Department of Sociology,
University of Oxford,
42-43 Park End Street,
Oxford,
OX1 1JD,
United Kingdom}
\address[M. Y. C. Chong]{Nuffield College,
University of Oxford,
New Road, Oxford,
OX1 1NF,
United Kingdom}
\address[M. Y. C. Chong]{Leverhulme Centre for Demographic Science,
University of Oxford,
42-43 Park End Street,
Oxford,
OX1 1JD,
United Kingdom}
\email{michael.chong@sociology.ox.ac.uk}

\author[T. Hibi]{Takayuki Hibi}
\address[T. Hibi]{Department of Pure and Applied Mathematics, Graduate School
of Information Science and Technology, The University of Osaka, Suita, Osaka
565-0871, Japan}
\email{hibi@math.sci.osaka-u.ac.jp}

\author[N. Ragunathan]{Naveena Ragunathan}
\address[N. Ragunathan]{Department of Mathematics and Statistics\\
McMaster University, Hamilton, ON, L8S 4L8, Canada}
\email{ragunatn@mcmaster.ca}
 
\author[A. Van Tuyl]{Adam Van Tuyl}
\address[A. Van Tuyl]{Department of Mathematics and Statistics\\
McMaster University, Hamilton, ON, L8S 4L8, Canada}
\email{vantuyla@mcmaster.ca}


\begin{abstract}
We study a family of positive weighted well-covered
graphs, which we call levelable graphs, that are 
related to a construction
of level artinian rings in commutative algebra.  
A graph $G$ is levelable if there exists a weight function
with positive integer values 
on the vertices of $G$ such that $G$ is well-covered with
respect to this weight function. In other words, the sum of the weights in
any maximal independent set of vertices of $G$ is the same.  
We describe some of the basic properties of levelable 
graphs and classify the levelable
graphs for some families of graphs, e.g., trees,
cubic circulants, Cameron--Walker graphs. 
We also explain the connection between levelable graphs
and a class of level artinian rings. Applying a result
of Brown and Nowakowski about weighted well-covered graphs, we 
show that for ``most graphs'', 
that is, random graphs $G$ for some constant
probability $p$,
the edge ideal $I(G)$ is  not Cohen--Macaulay.
\end{abstract}
 
\maketitle


\section{Introduction}

In this paper we introduce levelable graphs, a family of 
weighted well-covered graphs
that can be used to construct  level artinian rings, objects 
of great interest in commutative algebra.   
Throughout this paper $G = (V,E)$ denotes a finite simple graph on the 
vertex set $V = V(G) = \{x_1,\ldots,x_n\}$ and with edge set $E=E(G)$. 
A subset $W$ of $V$ is an {\it independent} set if 
$e \not\subseteq W$ for all edges $e \in E$. We say $W$ is a {\it maximal 
independent} set if no independent set of $G$ strictly includes $W$.  Let ${\rm MaxInd}(G)$ denote
the set of maximal independent sets of $G$.
In 1970 Plummer \cite{P1970} called a
graph {\it well-covered}
if every element of ${\rm MaxInd}(G)$ has the same cardinality. 

Weighted well-covered graphs, a generalization
of well-covered graphs, were introduced in 1998 by 
Caro, Ellingham, and Ramey \cite{CER1998}.  For a fixed
field $F$, a {\it weight function} on $G$ is a
function $w:V(G) \rightarrow F$ that assigns to vertex
$x_i \in V(G)$ a weight $c_i$.  Note that we can
view any $(c_1,\ldots,c_n) \in F^{|V|}$ as a weight
function on $G$.  A graph $G$ with a weight function
$(c_1,\ldots,c_n)$ is a {\it weighted well-covered graph} if
there is a $c\in F$ such that 
\[
\sum_{x_i \in  W} c_i = c 
~\mbox{for all $W \in {\rm MaxInd}(G)$}.
\]
In this case, $c$ is called the {\it independence weight}
of $G$.
In this language, a graph is well-covered if $G$ 
is a weighted well-covered graph with respect to the weight
function $(1,\ldots,1)$, and the independence
weight is the common cardinality of the 
maximum independent sets. 
Associated to
any graph $G$ is the set ${\rm WCW}(G)$ which consists
of all weight functions $(c_1,\ldots,c_n)$ that make
$G$ a weighted well-covered graph. In fact, the set ${\rm WCW}(G)$ is a vector space since 
it is closed under addition
and scalar multiplication.
Understanding the 
structure of ${\rm WCW}(G)$ has
been the focus of one branch of research surrounding weighted
well-covered graphs; see, for example, \cite{BKMUV2014,BN2005,BN2007,CER1998,LT2011,LT2015,LT2015a,T2022}.

When the field $F$ is $\mathbb{R}$ or $\mathbb{Q}$,
Caro, {\it et al.} \cite{CER1998} defined a {\it positive well-covered weighting} to be a weight function $(c_1,\ldots,c_n)$ 
that not only makes $G$ weighted well-covered, but $c_i > 0$
for all $i$.  As they remark in \cite[p. 653]{CER1998}:
\begin{quote}
We would guess that determining whether a positive
well-covered weighting exists is, in general, a difficult problem.
\end{quote}
We are interested in a variation of this 
``difficult problem'' by also requiring the $c_i$'s in the
weight function to
be positive integers. To this end, we make the following
definition:

\begin{definition}\label{defn.levelable}
A graph  $G$ is a {\it levelable graph} if there exists a 
weight function \[(c_1,\ldots,c_n) \in \mathbb{N}^n_{>0}\]
that makes $G$ a weighted well-covered graph.  In this
case, we say $G$ is a levelable graph with respect
to the weight function $(c_1,\ldots,c_n)$.  
\end{definition}

Our choice of name is inspired by the fact that a levelable
graph allows one to construct a level graded artinian ring.
Level rings were introduced by Stanley \cite{Stan77}
almost 50 years ago as a class of rings that sits 
``between'' Gorenstein rings and Cohen--Macaulay
rings. Understanding the properties of level rings has inspired
a significant amount of research, including 
\cite{GHMS2007,HSZ, HLO, HH06, H88, H89, H92, MNZ},
in commutative algebra.  
By applying a result of Van Tuyl and Zanello \cite{VTZ2010}, 
we have the  following link between graph theory 
and commutative algebra:

\begin{theorem}\label{maintheorem}
    Let $G = (V,E)$ be a graph,  and let $I(G)$ be the edge ideal of
    $G$ in the ring $R = \mathbb{K}[x_1,\ldots,x_n]$.  Then
    $G$ is a levelable graph with respect to the
    weight function $(c_1,\ldots,c_n)$ if and only if
    \[R/(I(G)+\langle x_1^{c_1+1},\ldots,x_n^{c_n+1}\rangle)\]
    is a level graded artinian ring.
\end{theorem}

\noindent
This connection with combinatorial commutative algebra,
as well as the ``difficult problem'' mentioned above, provides
motivation to classify or find families of levelable graphs.

It is immediate that levelable graphs exist -- 
all well-covered graphs are levelable with weight function $(1,\ldots,1).$  The goal of this paper is to identify 
other families of levelable graphs.  
Accordingly, we 
classify all the levelable graphs among 
the following families:
\begin{enumerate}
    \item[$\bullet$] Complete multipartite graphs (see Corollary \ref{cor:multipartite}),
    \item[$\bullet$] Cameron--Walker graphs (see Theorem \ref{theorem:CW}),
    \item[$\bullet$] Co-chordal graphs (see Theorem \ref{thm.cochordal}),
    \item[$\bullet$] Trees (see Theorem \ref{thm.trees}), and
    \item[$\bullet$] Cubic circulant graphs (see Theorems \ref{thm.circulant1} and \ref{thm.circulant2}).
\end{enumerate}
We also provide constructions of levelable graphs,
and obstructions that prevent $G$ from being levelable.

As an interesting by-product of
this new connection between weighted well-covered graphs
and graded artinian level rings, we can use results
about the vector space ${\rm WCW}(G)$ to deduce some
results in combinatorial commutative algebra.  
In particular, Brown and Nowakowski \cite{BN2007}
showed that for a random graph (suitably defined),
the expected dimension of ${\rm WCW}(G)$ is zero.  
This implies that for a random graph $G$,
there is no choice of positive integers $(a_1,\ldots,a_n)$
that makes  \[R/(I(G)+\langle x_1^{a_1},\ldots,x_n^{a_n}\rangle)\]
a level graded artinian
ring (see Corollary \ref{cor.neverlevel}).  Moreover, 
Brown and Nowakowski's result allows us to show
that for a  random graph $G$, $R/I(G)$ is not
Cohen--Macaulay (see Corollary \ref{cor.nevercm}),
complementing results of 
Dochtermann--Newman \cite[Corollary 1.4]{DN2023} and
Erman-Yang \cite[Corollary 7.1]{EY2018}.  

We have structured this paper so that the
first part focuses on the properties of levelable graphs.  
Only in the last section do we introduce the relevant 
commutative algebra in order to keep the
background on algebra to a minimum.
In Section 2 we include some basic
properties of levelable graphs and provide 
a useful lemma that
describes an obstruction to a graph being levelable.  Section 3
includes some constructions of levelable graphs
from a given levelable graph.  
The sections that follow
classify the levelable graphs among some families of graphs.
In the final section, we explain the connection to 
commutative algebra, which motivated our 
original interest in these
graphs, and explain how a result
of Brown and Nowakowski has consequences for combinatorial 
commutative algebra.


\section{Basic properties of levelable graphs}

We begin by exhibiting some of the basic properties and families
of levelable graphs.  We first note that we can
restrict
to {\it connected graphs}, where by definition there is a path in $G$ that connects any two distinct
vertices.

\begin{lemma}\label{lemma.connected}
Let $G$ and $H$ be disjoint graphs.  Then $G \cup H$ is levelable if
and only if $G$ and $H$ are levelable.
\end{lemma}

\begin{proof}
Let $G$ be a graph 
on $V(G) = \{x_1,\ldots,x_n\}$ and
$H$ a graph on $V(H) = \{y_1,\ldots,y_m\}$.  
Because the graphs $G$ and $H$ are 
disjoint, we have
\[{\rm MaxInd}(G \cup H) = \{W \cup Y ~|~ W \in {\rm MaxInd}(G),~ 
Y \in {\rm MaxInd}(H)\}.\]

{\bf (``If'')} 
Suppose that $G$ and $H$ are levelable with respect to 
the weight functions $(c_1,\ldots,c_n)$ and 
$(d_1,\ldots,d_m)$, respectively.
Let $c$ be the independence weight of $G$ and $d$ the 
independence weight of $H$.
We claim that $G \cup H$ is a levelable graph
with respect to the weight function
$(e_1,\ldots,e_n,e_{n+1},\ldots,e_{n+m}) = (c_1,\ldots,c_n,d_1,\ldots,d_m)$.
Indeed, for any $Z \in {\rm MaxInd}(G \cup H)$,
we have $Z = W \cup Y$ with $W \cap Y = \emptyset$ and
$W \in {\rm MaxInd}(G)$ and $Y \in {\rm MaxInd}(H)$. Thus
\[\sum_{z_i \in Z} e_i = \sum_{x_i \in W} c_i + \sum_{y_j \in Y} d_j =
c+d.\]

{\bf (``Only If'')} 
Suppose that 
$G \cup H$ is levelable with respect to
the weight function \[(e_1,\ldots,e_n,e_{n+1},\ldots,e_{n+m})\]
with independence weight $e$.
We claim $G$ and $H$ are levelable with respect
to the weight functions
$c = (e_1,\ldots,e_n)$ and 
$d = (e_{n+1},\ldots,e_{n+m})$, respectively.  Fix a
maximal independent set $Y$ of $H$ and consider any
$W_i \in {\rm MaxInd}(G)$.  
Then $W_i \cup Y \in {\rm MaxInd}(G \cup H)$ and so
\[\sum_{z_j \in W_i \cup Y} e_j = \sum_{z_j \in W_i} e_j + 
\sum_{z_j \in Y} e_j = e.\]
Rearranging this equation, we have 
\begin{equation}\label{sum}
\sum_{z_j \in W_i}e_j = e - \sum_{z_j \in Y} e_j.
\end{equation}
Note that since $W_i$ is a subset of $\{x_1,\ldots,x_n\}$, the 
$e_j$'s that appear in $\sum_{z_j \in W_i}e_j$ only appear among
$\{e_1,\ldots,e_n\}$.  Because \eqref{sum} is true for 
all maximal independent sets of $G$, $(e_1,\ldots,e_n)$ is a
weight function that makes $G$ a levelable  
graph with independence weight $e - \sum_{z_j \in Y} e_j$.  

By swapping the roles of $G$ and $H$, the same proof also shows
that $H$ is levelable.
\end{proof}

As noted in the introduction, all well-covered graphs
are levelable.  The following lemma allows us to deduce that
all complete multipartite graphs, which are not necessarily well-covered,
are levelable. Consequently, the set of levelable graphs
is strictly larger than the class of well-covered graphs.

\begin{lemma}\label{lem.pairwise} Suppose $G$ is a graph such
that the elements of  ${\rm MaxInd}(G)$ are pairwise 
disjoint.  Then $G$ is a levelable graph.
\end{lemma}

\begin{proof}
    Suppose ${\rm MaxInd}(G) = \{W_1,\ldots,W_t\}$
    with $d_i = |W_i|$ for $i =1,\ldots,t$.
    Since every vertex $x_i$ appears in some $W_i$
    and because the $W_i$'s are pairwise disjoint, we have that
    $W_1 \cup \cdots \cup W_t$ is a partition of $V(G)$.
    Thus, after relabeling, we can
    assume that 
    \[W_i = \{x_{d_1+d_2+ \cdots + d_{i-1}+1},\ldots,
    x_{d_1+d_2+\cdots +d_i}\} ~~\mbox{for $i=1,\ldots,t$},\] where $d_{0}=0.$
    
    Set $d = {\rm lcm}(d_1,\ldots,d_t)$.  A straightforward
    calculation will now show $G$ is a levelable
    graph with respect to the weight function
    \[(\underbrace{d/d_1,\ldots,d/d_1}_{d_1},\underbrace{d/d_2,\ldots,d/d_2}_{d_2},\ldots, 
    \underbrace{d/d_t,\ldots,d/d_t}_{d_t}).\]
\end{proof}

For integers $d, a_1,\ldots, a_d \geq 1$,
the {\it complete $d$-partite graph} $K_{a_1,\ldots,a_d}$
is the graph with vertices 
$V = \{x_{i,j} ~|~ 1 \leq i \leq d, ~1 \leq j \leq a_i\}$ and edges of the form
$\{x_{i,j},x_{k,l}\}$ for all $1 \leq i < k \leq d$.  
Note that with this notation, the well-known
complete graph $K_d$ is the graph $K_{1,\ldots,1}$.\footnote[1]{This is an abuse 
of notation, since our notation for a complete multipartite graph would imply that $K_d$ should
be the  $1$-partite graph consisting
of $d$ isolated vertices.  Moving forward,
we will use the standard notation of $K_d$ for 
the complete graph.} Since
the maximal independent sets of $K_{a_1,\ldots,a_d}$ 
are the pairwise disjoint sets
$W_i = \{x_{i,1},\ldots,x_{i,a_i}\}$ for $i=1,\ldots,d$, we have
the following corollary.
\begin{corollary}\label{cor:multipartite}
    For all $d, a_1,\ldots,a_d \geq 1$, the
    complete $d$-partite graph $K_{a_1,\ldots,a_d}$ 
    is levelable.
\end{corollary}

The {\it independence number} of $G$, denoted $\alpha(G)$, is
the size of the largest maximal independent set of $G$.   If
$\alpha(G)$ is small, then $G$ must also be levelable.

\begin{theorem}\label{thm.smallalapha}
    If $G$ is a graph on $\{x_1,\dots,x_n\}$ with $\alpha(G) \leq 2$, then 
    $G$ is levelable.
\end{theorem}

\begin{proof}
    If $\alpha(G) =1$, then $G = K_n$ is the 
    complete $n$-partite graph (or simply, complete
    graph) and is levelable by Theorem \ref{cor:multipartite}.

    Suppose that $\alpha(G) =2$.  If all $W \in {\rm MaxInd}(G)$
    have $|W|=2$, then $G$ is well-covered, and thus levelable.
    So, suppose that $|W| =1$ for at least
    one $W \in {\rm MaxInd}(G)$.  Let
    ${\rm MaxInd}(G) = \{W_1,\ldots,W_t\}$. After relabeling we can assume
    that $|W_1| = \cdots = |W_s| =1$ and $|W_{s+1}|=\cdots=
    |W_t|=2$.  Note that if $|W_j|=1$, then $W_j \cap W_k 
    = \emptyset$ for all $j \neq k$.  So, we can relabel
    again so that $W_i = \{x_i\}$ for $i=1,\ldots,s$,
    and $W_{s+1} \cup \cdots \cup W_t = \{x_{s+1},\ldots,x_n\}$.
    Then $G$ is levelable with respect to the weight
    function
    \[c = (\underbrace{2,\ldots,2}_s,\underbrace{1,\ldots,1}_{n-s})\]
    since $\sum_{x_j\in W_i} c_j = 2$ if $|W_i| =1$ and
    $\sum_{x_j \in W_i} c_j = 2$ if $|W_i|=2$. 
\end{proof}

We let $C_n$ denote the {\it $n$-cycle} on $n$-vertices with
$n \geq 3$. For any
graph $G$, we write $G^c$ for the {\it graph complement} of $G$,
that is, the graph with the same vertex set as $G$, but
edge set $\{\{x_i,x_j\} : \{x_i,x_j\} \not\in E(G)\}$.
When $n =3 $, $C_n^c$ is three disjoint vertices,
and thus clearly levelable. On the other hand,
$\alpha(C_n^c) = 2$ if $n \geq 4$.
Consequently, we have the following corollary:

\begin{corollary}\label{cor.complementcycle}
For any $n \geq 3$, $C_n^c$ is levelable.
\end{corollary}

The next example exhibits a 
non-levelable graph
with $\alpha(G) =3$. In fact, this
example illustrates that levelable
graphs are not preserved by taking induced subgraphs.

\begin{example}\label{ex.P5}
Let $P_n$ denote the {\it path graph} 
of length $n-1$ on $n \geq 2$ vertices,
that is, the graph with vertex set $\{x_1,\ldots,x_n\}$
and edge set $\{\{x_i,x_{i+1}\} ~:~ 1 \leq i \leq n-1\}$.
We show that $P_5$ (see Figure \ref{fig:P5}) is not levelable.  
We see that ${\rm MaxInd}(P_5) = 
\{\{x_1,x_3,x_5\}, \{x_1,x_4\},\{x_2,x_4\},\{x_2,x_5\}\}$, and thus $\alpha(G) =3$.
\begin{figure}[h]
\centering
\begin{tikzpicture}[scale=1.5]
\coordinate (x1) at (0,0){};
\coordinate (x2) at (1,0){};
\coordinate (x3) at (2,0){};
\coordinate (x4) at (3,0){};
\coordinate (x5) at (4,0){};
\fill(x1)circle(0.7mm);
\fill(x2)circle(0.7mm);
\fill(x3)circle(0.7mm);
\fill(x4)circle(0.7mm);
\fill(x5)circle(0.7mm);
\draw(x1)--(x2)--(x3)--(x4)--(x5);
\draw(x1)node[below=1mm]{{$x_1$}};
\draw(x2)node[below=1mm]{{$x_2$}};
\draw(x3)node[below=1mm]{{$x_3$}};
\draw(x4)node[below=1mm]{{$x_4$}};
\draw(x5)node[below=1mm]{{$x_5$}};
\end{tikzpicture}
\caption{The path $P_5$ which is non-levelable} \label{fig:P5}
\end{figure}
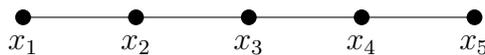 
Suppose that a weight function $(c_1,\ldots,c_5)$ makes $P_5$ levelable.  Then 
\[(c_1+c_3+c_5)+(c_2+c_4) = (c_1+c_4)+(c_2+c_5).\]
Consequently $c_3 =0$, contradicting the fact that
the entries of $(c_1,\ldots,c_5)$ are all positive.  

The graph given
in Figure \ref{fig:P5subgraph} 
is a well-covered graph that has $P_5$
as an induced subgraph.  
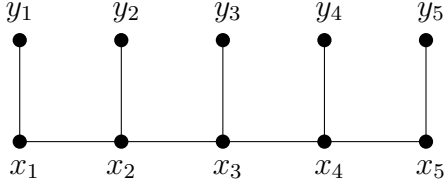
\begin{figure}[ht]
\centering
\begin{tikzpicture}[scale=1.35]
\coordinate (x1) at (0,0){};
\coordinate (x2) at (1,0){};
\coordinate (x3) at (2,0){};
\coordinate (x4) at (3,0){};
\coordinate (x5) at (4,0){};
\coordinate (y1) at (0,1){};
\coordinate (y2) at (1,1){};
\coordinate (y3) at (2,1){};
\coordinate (y4) at (3,1){};
\coordinate (y5) at (4,1){};
\fill(x1)circle(0.7mm);
\fill(x2)circle(0.7mm);
\fill(x3)circle(0.7mm);
\fill(x4)circle(0.7mm);
\fill(x5)circle(0.7mm);
\fill(y1)circle(0.7mm);
\fill(y2)circle(0.7mm);
\fill(y3)circle(0.7mm);
\fill(y4)circle(0.7mm);
\fill(y5)circle(0.7mm);
\draw(x1)--(x2)--(x3)--(x4)--(x5);
\draw(x1)--(y1);
\draw(x2)--(y2);
\draw(x3)--(y3);
\draw(x4)--(y4);
\draw(x5)--(y5);
\draw(x1)node[below=1mm]{{ $x_1$}};
\draw(x2)node[below=1mm]{{$x_2$}};
\draw(x3)node[below=1mm]{{ $x_3$}};
\draw(x4)node[below=1mm]{{ $x_4$}};
\draw(x5)node[below=1mm]{{ $x_5$}};
\draw(y1)node[above=1mm]{{$y_1$}};
\draw(y2)node[above=1mm]{{ $y_2$}};
\draw(y3)node[above=1mm]{{ $y_3$}};
\draw(y4)node[above=1mm]{{ $y_4$}};
\draw(y5)node[above=1mm]{{ $y_5$}};
\end{tikzpicture}
\caption{
A levelable graph with $P_5$ as an induced
subgraph} \label{fig:P5subgraph}
\end{figure} 
Consequently, the levelable property
is not preserved by taking induced
subgraphs.
\end{example}

Embedded in Example \ref{ex.P5} is a condition on
the elements of ${\rm MaxInd}(G)$
that prevents $G$ from being levelable.  We make this condition
explicit in the next lemma.

\begin{lemma} \label{lem.noncondition}
    Suppose there exist $F_1, F_2, F_3, F_4 \in {\rm MaxInd}(G)$ such that $F_3 \cup F_4 \subsetneq F_1 \cup F_2$ and $F_3 \cap F_4 = \emptyset$. Then G is not levelable.
\end{lemma}

\begin{proof}
    Suppose towards a contradiction that ($c_1,\dots, c_n$) is a weight function for $G$ with independence weight $c$. Then $\sum_{x_j \in F_1} c_j + \sum_{x_j \in F_2} c_j  = \sum_{x_j \in F_3} c_j + \sum_{x_j \in F_4} c_j =2c$. Since $c_j >0$ for all $j$, and $F_3 \cup F_4 \subsetneq F_1 \cup F_2$, we must have 
    \[
    \sum_{x_{j} \in F_3 \cup F_4} c_j < \sum_{x_{j} \in F_1 \cup F_2} c_j \leq \sum_{x_{j} \in F_1} c_j + \sum_{x_{j} \in F_2} c_j = 2c.
    \]
    But then since $F_3 \cap F_4 = \emptyset$, we have 
    \[ 2c = \sum_{x_{j} \in F_3} c_j + \sum_{x_{j} \in F_4} c_j = \sum_{x_{j} \in F_3 \cup F_4} c_j < 2c,
    \]
    which is impossible.
\end{proof}

This next result, which
can also be deduced from
Theorem \ref{thm.trees} given later in the paper, demonstrates how to apply Lemma \ref{lem.noncondition}.

\begin{corollary}\label{cor.paths}
    The path graph $P_n$ is levelable
    if and only if $n\in\{2,3,4\}$.
\end{corollary}

\begin{proof}
{\bf (``If'')} 
    When $n=2$ or $4$, the graph $P_n$ is well-covered, and thus
    levelable.  When $n=3$, the weight function $(1,2,1)$
    makes $P_3$ levelable.

    {\bf (``Only If'')}
    Now suppose that $n \geq 5$.  Suppose that $n = 2m+1$
    is odd.  Then 
    \begin{eqnarray*}
    F_1 &= &\{x_1,x_3,\ldots,x_{2m-1},x_{2m+1}\}, ~~F_2 = \{x_2,x_4,\ldots,x_{2m}\},\\
    F_3 & =& \{x_1,x_4,x_6,\ldots,x_{2m}\} 
    ~~\mbox{and}~~
    F_4  =  \{x_2,x_5,x_7,\ldots,x_{2m+1}\}
    \end{eqnarray*}
    are all elements of ${\rm MaxInd}(G)$.  The result now
    follows from Lemma \ref{lem.noncondition}.

    Similarly, if $n =2m$ is even, we have $m \geq 3$, and
    we consider the sets:
    \begin{eqnarray*}
    F_1 &= &\{x_1,x_3,\ldots,x_{2m-1}\}, 
    F_2 = \{x_2,x_4,\ldots,x_{2m}\}, \\
    F_3 & =& \{x_1,x_4,x_6,\ldots,x_{2m}\} 
    ~~\mbox{and}~~
    F_4  =  \{x_2,x_5,x_7,\ldots,x_{2m-1}\}
    \end{eqnarray*}
    and again use Lemma \ref{lem.noncondition}.
\end{proof}

\begin{remark}
While our paper has focused on independent sets, 
we could have  used the notion of vertex covers 
to define levelable graphs. While we do not
use this reformulation, we have included it in
case others wish to make use of it.
A subset $W \subseteq V(G)$ is called a {\it vertex cover} if $e \cap W 
\neq \emptyset$ for all $e \in E(G)$.     A vertex cover is the complement
of an independent set, that is, $W$ is a vertex
cover of $G$ if and only
if $V\setminus W$ is an independent set.  Maximal independent sets correspond
to minimal vertex covers.  

If we let ${\rm MinVC}(G)$ be the collection of minimal vertex covers, then
the following definition for levelable graphs is equivalent to Definition \ref{defn.levelable}: a graph $G$ is levelable if there exists
a weight function $(d_1,\ldots,d_n) \in \mathbb{N}^n_{> 0}$ and some constant $d$ such that
\[\sum_{x_i \in W} d_i = d ~~\mbox{for all $W \in {\rm MinVC}(G)$}.\]
To see that these definitions are equivalent, set $D = d_1+\cdots+d_n$. Then
\[ \sum_{x_i \in V\setminus W} d_i = D - \sum_{x_i \in W} d_i = D-d.\]
But as $W$ runs through all elements of ${\rm MinVC}(G)$, 
$V\setminus W$ runs through all elements of ${\rm MaxInd}(G)$.  So
$G$ is levelable with respect to Definition \ref{defn.levelable}. By
simply reversing this argument, we can show Definition \ref{defn.levelable}
is equivalent to this reformulation.
\end{remark}


\section{Constructions} 
In this section, we introduce basic techniques to construct a new levelable graph from a levelable graph.

\begin{definition}
Let $G$ be a finite graph and fix a vertex $x$ of $G$.  

\begin{itemize}
    \item 
The graph $G^x$, called the {\it duplication} of $G$ at $x$, is the finite graph obtained by adding a new vertex $y$ to $G$ with
\[
E(G^x) = E(G) \cup \{\{y, b\} : \{x, b\} \in E(G)\}.
\]
\item
The graph $G^{[x]}$, called the {\it expansion} of $G$ at $x$, is the finite graph obtained by adding a new vertex $y$ to $G$ with 
\[
E(G^{[x]}) = E(G^x) \cup \{\{x,y\}\}.
\]  
\end{itemize}
\end{definition}

\begin{theorem}
Suppose that $G$ is a levelable graph and fix a vertex $x$ of $G$.  Then both $G^x$ and $G^{[x]}$ are levelable.  
\end{theorem}

\begin{proof}
Let $G$ be a levelable graph on $V = \{x_1, \ldots, x_n\}$ with respect to a weight function $(c_1, \ldots, c_n)$.  Set $x=x_n$ and $y=x_{n+1}$.

Let $W$ be a maximal independent set  of $G^x$.  If $x \not\in W$, then $W$ is a maximal independent set of $G$.  If $x \in W$, then $y \in W$ and $W \setminus \{y\}$ is a maximal independent set of $G$.  Hence $G^x$ is levelable with respect to the weight function
\[
(2c_1,\cdots, 2c_{n-1}, c_n, c_n).
\]

A maximal independent set $W$ of $G^{[x]}$ is either a maximal independent set of $G$ or $(W \setminus \{x\}) \cup \{y\}$, where $W$ is a maximal independent set of $G$ with $x \in W$.
Hence $G^{[x]}$ is levelable with respect to the weight function
\[
(c_1,\cdots, c_{n-1}, c_n, c_n).
\]
\end{proof}

\begin{example}
The pentagon is levelable since it is a well-covered graph.  Then repeated applications of the duplication technique shows that the graph of Figure \ref{fig:Pentagon} is levelable.
\end{example}
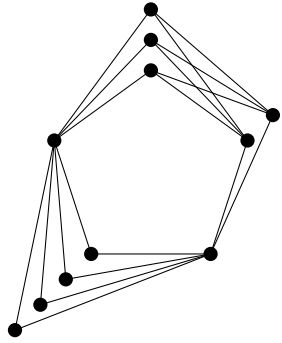
\begin{figure}[h]
\centering
\begin{tikzpicture}[scale=1.35]
\coordinate (A) at (0,1);
\coordinate (B) at ({cos(9*pi/10 r)},{sin(9*pi/10 r)});
\coordinate (C) at ({cos(13*pi/10 r)},{sin(13*pi/10 r)});
\coordinate (D) at ({cos(17*pi/10 r)},{sin(17*pi/10 r)});
\coordinate (E) at ({cos(21*pi/10 r)},{sin(21*pi/10 r)});
\coordinate (F) at ($(A)+(0,0.3)$);
\coordinate (G) at ($(F)+(0,0.3)$);
\coordinate (H) at ($(E)+(0.25,0.25)$);
\coordinate (I) at ($(C)-(0.25,0.25)$);
\coordinate (J) at ($(I)-(0.25,0.25)$);
\coordinate (K) at ($(J)-(0.25,0.25)$);
\fill(A)circle(0.7mm);
\fill(B)circle(0.7mm);
\fill(C)circle(0.7mm);
\fill(D)circle(0.7mm);
\fill(E)circle(0.7mm);
\fill(F)circle(0.7mm);
\fill(G)circle(0.7mm);
\fill(H)circle(0.7mm);
\fill(I)circle(0.7mm);
\fill(J)circle(0.7mm);
\fill(K)circle(0.7mm);
\draw (A)--(B)--(C)--(D)--(E)--cycle;
\draw (B)--(G)--(H)--(D)--(K)--cycle;
\draw (B)--(F)--(H);
\draw (G)--(E);
\draw (A)--(H);
\draw (F)--(E);
\draw (B)--(I)--(D);
\draw (B)--(J)--(D);
\end{tikzpicture}
\caption{A pentagon with duplications.} \label{fig:Pentagon}
\end{figure} 

\begin{corollary}
    Let $G$ be a levelable graph on $\{x_1, \ldots, x_n\}$ with respect to a weight function $(c_1, c_2, \ldots, c_n)$.  Then for any positive
    integers $r_1,\ldots,r_n$, there is a graph $G'$ which is levelable with respect to the weight function 
    \begin{equation}
    \label{WWWWW}
    (\underbrace{c_1, c_1, \ldots, c_1}_{r_1}, 
    \underbrace{c_2, c_2, \ldots, c_2}_{r_2}, 
    \ldots, 
    \underbrace{c_n, c_n, \ldots, c_n}_{r_n}). 
        \end{equation}
\end{corollary}

\begin{proof}
Let $r_i$ be the number of $c_i$ in (\ref{WWWWW}).  Then doing
$r_i -1$ expansions of each vertex $x_i$ 
yields the desired graph $G'$.   
\end{proof}

\begin{remark}
Let $G$ be a graph consisting of $n$ isolated vertices.  Since $\{x_1, \ldots, x_n\}$ is the unique maximal independent set of $G$, it follows that $G$ is levelable with respect to any weight function $(c_1, \ldots, c_n)$.  In particular, given any sequence $(c_1, \ldots, c_n)$ of positive integers, there is a levelable graph on $\{x_1, \ldots, x_n\}$ with respect to $(c_1, \ldots, c_n)$. For connected graphs, there will be certain restrictions on what
sequences can arise. If $n=3$, then the connected graphs are the triangle and the path of length $2$.  Hence given a sequence $(c_1, c_2, c_3)$ of positive integers with $c_1 \leq c_2 \leq c_3$, there is a connected levelable graph on $\{x_1, x_2, x_3\}$ with respect to $(c_1, c_2, c_3)$ if and only if either $c_1=c_2=c_3$ or $c_3 = c_1 + c_2$.  It would, of course, be of interest to classify all possible weight functions for connected levelable graphs. Note that if $G$ is a levelable graph  on $\{x_1, \ldots, x_n\}$ with respect to a weight function $(c_1, \ldots, c_n)$, then $G$ is also levelable with respect to the weight function $(mc_1, \ldots, mc_n)$, where $m$ is a positive integer. Consequently, it suffices to classify the vectors $(c_1,\dots,c_n)$ with $\gcd(c_1,\dots,c_n)=1$.
\end{remark}

\begin{definition}
\label{defGH}
Let $G$ be a finite connected graph on $V(G) = \{x_1, \ldots, x_n\}$.  Given finite graphs $H_1,\ldots, H_n$ with each $V(H_i) = \{z^{(i)}_1, \ldots, z^{(i)}_{r_i} \}$, where $V(G), V(H_1), \ldots, V(H_n)$ are pairwise disjoint, we introduce the finite graph $G(H_1, \ldots, H_n)$ on $V(G) \cup V(H_1) \cup \cdots \cup V(H_n)$ whose set of edges is
\[
E(G) \cup E(H_1) \cup \cdots \cup E(H_n) \cup \left(\,\bigcup_{\substack{1 \leq i \leq n \\ 1 \leq j \leq r_i}}\{\{x_i, z^{(i)}_j \}\}\right).
\] 
\end{definition}

\begin{theorem}
\label{Sapporo}
With the notation as in Definition \ref{defGH}, the graph $G(H_1, \ldots, H_n)$ is levelable if and only if each $H_i$ is levelable. 

In particular, if each $H_i$ is a complete graph $K_{r_i}$, then $G(H_1, \ldots, H_n)$ is levelable.  
\end{theorem}

\begin{proof}
Let $\Gamma = G(H_1, \ldots, H_n)$.  A maximal independent set of $\Gamma$ consists of an independent set $W_0$ of $G$ together with maximal independent sets $W_i$ of $H_i$ for each $i$ such that $x_i \not\in W_0$.

{\bf (``If'')}  
Suppose that each $H_i$ is levelable with respect to a weight function \[(c(z^{(i)}_1), \ldots, c(z^{(i)}_{r_i})).\]  Let $d_i = \sum_{z^{(i)}_j \in W} c(z^{(i)}_j)$ be its independence weight, where $W$ is a maximal independent set of $H_i$.  Let $c(x_i) = d_i$ for $1 \leq i \leq n$. It can be checked that $\Gamma$ is levelable with respect to the weight function 
\begin{eqnarray}
    \label{ccccc}
(c(x_1), \ldots, c(x_n), c(z_1^{(1)}), \ldots, c(z_{r_n}^{(n)})).
\end{eqnarray}

{\bf (``Only If'')}
Suppose that $\Gamma$ is levelable with respect to a weight function of the form (\ref{ccccc}). Let $W$ be a maximal independent set of $\Gamma$ with $x_i \not\in W$ and $W_i$ a maximal independent set of $H_i$ with $W_i \subseteq W$.  Let $W'_i$ be a maximal independent set of $H_i$ with $W_i \neq W_i'$.  Since $(W \setminus W_i) \cup W'_i$ is a maximal independent set of $\Gamma$, it follows that $H_i$ is levelable with respect to the weight function $(c(z^{(i)}_1), \ldots, c(z^{(i)}_{r_i}))$.
\end{proof}

\begin{corollary}
\label{cor:AAA}
Given positive integers $c_1, \ldots, c_n$ and
positive integers $r_1,\ldots,r_n$ with $r_i \geq 2$ for all $i$,
there is a connected levelable graph on 
$r_1+\cdots +r_n$ vertices
with respect to the weight function 
\[
(\underbrace{c_1, \ldots,c_1}_{r_1}, \underbrace{c_2, 
\ldots, c_2}_{r_2}, \ldots, \underbrace{c_n, \ldots, c_n}_{r_n}).
\]
\end{corollary}

\begin{proof}
Let $P_n$ be the path on $\{x_1, \ldots, x_n\}$ and let
$K_{r_i-1}$ be the  complete graph on the vertex set
$\{z_1^{(i)},\ldots,z_{r_i-1}^{(i)}\}$.
By Theorem \ref{Sapporo}, the graph $G = P_n(K_{r_1-1},\ldots,K_{r_n-1})$
is levelable.
A maximal independent set of $G$ is of the form $A \cup B$, 
where $A$ is an independent set of $P_n$, $B$ is a subset
of $\bigcup_{i=1}^n\{z_1^{(i)},\ldots,z_{r_i-1}^{(i)}\}$,
and for each $x_i \not\in A$,  $|B \cap \{z_1^{(i)},\ldots,z_{r_i-1}^{(i)}\}|=1$,
that is, $B$ contains exactly one of the vertices of $K_{r_i-1}$.
In particular, $G$ is well-covered.  Because 
each maximal independent set of $G$ contains exactly
one vertex of $\{x_i,z_1^{(i)},\ldots,z_{r_i-1}^{(i)}\}$ for
each $i$,
the graph $G$ is then levelable with respect to the weight function
\begin{multline*}
(c(x_1), c(z_1^{(1)}),\ldots,c(z_{r_1-1}^{(1)}),c(x_2),
c(z_1^{(2)}),\ldots,c(z_{r_2-1}^{(2)}),\ldots,
c(x_n),
c(z_1^{(n)}),\ldots,c(z_{r_n-1}^{(n)})) \\
= 
(\underbrace{c_1, \ldots,c_1}_{r_1}, \underbrace{c_2, 
\ldots, c_2}_{r_2}, \ldots, \underbrace{c_n, \ldots, c_n}_{r_n}).
\end{multline*}
\end{proof}


\begin{corollary}
\label{cor:BBB}
Given positive integers $c_1, \ldots, c_n$, there is a connected levelable graph on $V=\{x_1, \ldots, x_n, \ldots, x_N\}$ with respect to the weight function
\[
(c_1, \ldots, c_n, 1, \ldots, 1) \in \ZZ^{N},
\]
 where 
\[
N = n + \sum_{i=1}^{n} c_i.
\]
\end{corollary}

\begin{proof}
Let $P_n$ be the path on $\{x_1, \ldots, x_n\}$.  We introduce the graph $G$ on 
\[
\{x_1, \ldots, x_n\} \cup \{y_j^{(i)} : 1 \leq i \leq n, 1 \leq j \leq c_i\}
\]
which is obtained by adding edges $\{x_i, y_1^{(i)}\}, \ldots, \{x_i, y_{c_i}^{(i)}\}$ to $P_n$.  Theorem \ref{Sapporo} says that $G$ is levelable.  A maximal independent set of $G$ is of the form $A \cup B$, where $A$ is an independent set of $P_n$ and $B= \{y_j^{(i)} : x_i \not\in A, 1 \leq j \leq c_i\}$.  Hence $G$ is levelable with respect to the weight function
\[
(c(x_1), \ldots, c(x_n), c(y^{(1)}_1), \ldots, c(y^{(n)}_{c_n})) = (c_1, \ldots, c_n, 1, 1, \ldots, 1) \in \ZZ^N. 
\]
\end{proof}

\section{Cameron--Walker graphs}

We now turn to the discussion of the problem of which Cameron--Walker graphs are levelable.  The reason why Cameron--Walker graphs are important is that the regularity of the 
symbolic powers of the edge ideal of a Cameron--Walker graph can be computed explicitly \cite{S}. Recall that a Cameron--Walker graph \cite{CW2005,HHKO} is a finite graph consisting of a connected bipartite graph with vertex partition $U \cup V$ such that there is at least one leaf (or pendant edge) attached to each $x \in U$ and that there may be possibly some pendant triangles attached to each $y \in V$.  A vertex $y \in V$ is called {\it exceptional} if there is no pendant triangle attached to $y$.  It will be assumed that, for each exceptional $y \in V$, there exists $x, x' \in U$ with $x \neq x'$ for which $\{x,y\}, \{x',y\} \in E(G)$. 

\begin{example}
The Cameron--Walker graph $G$ of Figure \ref{fig:CW} is not levelable.  To see why this is true, suppose that $G$ is levelable with respect to a weight function $(c_1, \ldots, c_{10})$.  Since $\{x_1, x_2, x_3, x_7, x_8\}$ and $\{x_1, x_2, x_6, x_7\}$ are maximal independent sets of $G$, one has $c_3 + c_8 = c_6$.  Furthermore, since $\{x_3, x_4, x_5, x_9\}$ and $\{x_4, x_5, x_6, x_9\}$ are maximal independent sets of $G$, one has $c_3 = c_6$.  Hence $c_8 = 0$, a contradiction.
\begin{figure}[h]
\centering
\begin{tikzpicture}[scale=1.35]
\coordinate (x1) at (0,3){};
\coordinate (x2) at (1,3){};
\coordinate (x3) at (2,3){};
\coordinate (x4) at (0,2){};
\coordinate (x5) at (1,2){};
\coordinate (x6) at (2,2){};
\coordinate (x7) at (0.6,1){};
\coordinate (x8) at (1.5,1){};
\coordinate (x9) at (.1,0){};
\coordinate (x10) at (1,0){};
\fill(x1)circle(0.7mm);
\fill(x2)circle(0.7mm);
\fill(x3)circle(0.7mm);
\fill(x4)circle(0.7mm);
\fill(x5)circle(0.7mm);
\fill(x6)circle(0.7mm);
\fill(x7)circle(0.7mm);
\fill(x8)circle(0.7mm);
\fill(x9)circle(0.7mm);
\fill(x10)circle(0.7mm);
\draw(x1)--(x4)--(x7)--(x9)--(x10)--(x7)--(x5)--(x2);
\draw(x3)--(x6)--(x8)--(x5);
\draw(-0.25,3)node{{$x_1$}};
\draw(0.75,3)node{{$x_2$}};
\draw(2.25,3)node{{ $x_3$}};
\draw(-0.25,2)node{{ $x_4$}};
\draw(0.75,2)node{{ $x_5$}};
\draw(2.25,2)node{{ $x_6$}};
\draw(0.25,1)node{{ $x_7$}};
\draw(1.75,1)node{{ $x_8$}};
\draw(-0.25,0)node{{ $x_9$}};
\draw(1.3,0)node{{ $x_{10}$}};
\end{tikzpicture}
\caption{A non-levelable Cameron--Walker graph} \label{fig:CW}
\end{figure}
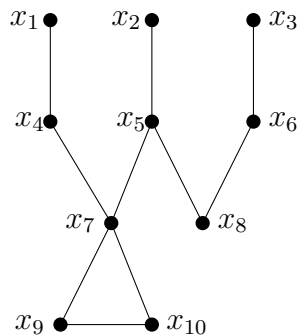 
\end{example}

We now come to a classification of levelable Cameron--Walker graphs.

\begin{theorem}
\label{theorem:CW}
A Cameron--Walker graph is levalable if and only if it has no exceptional vertex.
\end{theorem}

\begin{proof}
{\bf (``If'')}
This direction follows from Theorem \ref{Sapporo}.  

{\bf (``Only If'')}
Suppose that $G$ is a Cameron--Walker graph with an exceptional vertex.  Let $U = \{x_1, \ldots, x_a\}$ and $V = \{y_1, \ldots, y_b\}$.  Let $A_i$ denote the set of vertices $z$ of $G$ for which $\{x_i,z\}$ is a
leaf attached to $x_i$ for $1 \leq i \leq a$ and $B_j$ the set of edges $\{v,w\}$ for which $G|_{\{y_j,v,w\}}$ is a pendant triangle attached to $y_j$ for $1 \leq j \leq b$.  Each $A_i$ is nonempty and some $B_j$ might be empty.  Let $A_i = \{z_1^{(i)}, \ldots, z_{q_i}^{(i)}\}$ and $B_j = \{\{v_1^{(j)},w_1^{(j)}\}, \ldots, \{v_{r_j}^{(j)},w_{r_j}^{(j)}\}\}$.  Let $V'$ denote the set of exceptional vertices.  Without loss of generality, let $y_1$ be an exceptional vertex and $\{x_1, y_1\} \in E(G)$.  Now,
\[
(A_1 \cup \cdots \cup A_a) \cup V' \cup \left(\,\bigcup_{y_j \not\in V'}\{v_1^{(j)}, \ldots, v_{r_j}^{(j)}\}\right)
\]
is a maximal independent set of $G$.  Furthermore,
\[
\{x_1\} \cup (A_2 \cup \cdots \cup A_a) \cup \{y_j \in V': \{x_1,y_j\} \not\in E(G)\} \cup \left(\,\bigcup_{y_j \not\in V'}\{v_1^{(j)}, \ldots, v_{r_j}^{(j)}\}\right)
\]
is a maximal independent set of $G$.  Suppose that $G$ is levelable with respect to a weight function $(c(x_1), \ldots, c(y_b), c(z^{(1)}_1), \ldots)$.  It then follows that
\begin{eqnarray}
\label{aaaaa}
c(z^{(1)}_1)+\cdots+c(z^{(1)}_{q_1}) + \sum_{y_j \in V', \, \{x_1,y_j\} \in E(G)} c(y_j) = c(x_1).
\end{eqnarray}
On the other hand, 
\[
U \cup \left(\,\bigcup_{y_j \not\in V'}\{v_1^{(j)}, \ldots, v_{r_j}^{(j)}\}\right)
\]
is a maximal independent set of $G$.  Furthermore, by using the assumption that, for each $y \in V'$, there exists $x, x' \in U$ with $x \neq x'$ for which $\{x,y\}, \{x',y\} \in E(G)$, it follows that  
\[
A_1 \cup \{x_2, \ldots, x_a\} \cup \left(\,\bigcup_{y_j \not\in V'}\{v_1^{(j)}, \ldots, v_{r_j}^{(j)}\}\right)
\]
is a maximal independent set of $G$.  Hence
\begin{eqnarray}
\label{bbbbb}
c(x_1)=c(z^{(1)}_1)+\cdots+c(z^{(1)}_{q_1}).
\end{eqnarray}
Finally, by (\ref{aaaaa}) and (\ref{bbbbb}), one has
\[
\sum_{y_j \in V', \, \{x_1,y_j\} \in E(G)} c(y_j) = 0,
\]
a contradiction.  
\end{proof}

Suppose that $G$ is a Cameron--Walker graph with no exceptional vertex and work in the same notation as in the proof of Theorem \ref{theorem:CW}.  Then $G$ is levelable with respect to the weight function 
\[
(c(x_1), \ldots, c(y_b), c(z^{(1)}_1), \ldots) =(q_1, \ldots, q_a, r_1, \ldots r_b, 1, 1, \ldots, 1,1).
\]

\section{Chordal graphs}

In this section we turn our attention
to levelable chordal graphs.  Recall that a graph
is {\it chordal} if it is either a tree or all of its 
induced cycles have length three.   
In light of Example \ref{ex.P5}, some 
chordal graphs will not be levelable. Consequently, it is a reasonable project to classify all levelable chordal graphs.
As the first step in this project, we show that
the complement of any chordal graph is levelable, and
we classify the trees that are levelable.

Let $\Delta$ be a simplicial complex (the formal definitions are postponed until
Section 7). A vertex $x$ of $\Delta$ is called {\it free} if $x$ belongs to exactly one facet.  A facet $F$ of $\Delta$ is said to be a {\it leaf} of $\Delta$ if either 
$F$ is a unique facet of $\Delta$, or there exists a facet $G$ of $\Delta$ with $G \neq F$, called a {\it branch} of $F$, for which $H \cap F \subseteq G \cap F$ for all facets $H$ of $\Delta$ with $H \neq F$.  Each vertex belonging to $F \setminus G$ is free.  A labeling $F_1, \ldots, F_s$ of the facets of $\Delta$ is called a {\it leaf order} if, for each $1 < i \leq s$, the facet $F_i$ is a leaf of the subcomplex $\langle F_1, \ldots, F_{i}\rangle$.  (Recall that $\langle F_1, \ldots, F_{i}\rangle$ is the subcomplex of $\Delta$ consisting of those faces $F$ with $F \subset F_j$ for some $1 \leq j \leq i$.)  A simplicial complex possessing a leaf order is called a {\it quasi-forest}.  We refer the reader to \cite[Chapter 9]{HHgtm260} for more details.

\begin{figure}[h]
\centering
\begin{tikzpicture}[scale=1.35]
\coordinate (x1) at (0.5,1){};
\coordinate (x2) at (0,0){};
\coordinate (x3) at (1,0){};
\coordinate (x4) at (-0.5,-1){};
\coordinate (x5) at (0.5,-1){};
\coordinate (x6) at (1.5,-1){};
\fill(x1)circle(0.7mm);
\fill(x2)circle(0.7mm);
\fill(x3)circle(0.7mm);
\fill(x4)circle(0.7mm);
\fill(x5)circle(0.7mm);
\fill(x6)circle(0.7mm);
\draw(x1)--(x4)--(x6)--cycle;
\draw(x2)--(x3)--(x5)--cycle;
\draw(0.45,1.25)node{{  $x_1$}};
\draw(-0.32,0)node{{  $x_2$}};
\draw(1.25,0)node{{  $x_3$}};
\draw(-0.5,-1.25)node{{  $x_4$}};
\draw(0.5,-1.25)node{{  $x_5$}};
\draw(1.5,-1.25)node{{  $x_6$}};
\draw(0.45,0.3)node{{  $F_1$}};
\draw(0.45,-0.35)node{{  $F_2$}};
\draw(-0.05,-0.6)node{{  $F_3$}};
\draw(0.95,-0.6)node{{  $F_4$}};
\end{tikzpicture}
\caption{A leaf order} \label{fig:leaf}
\end{figure}
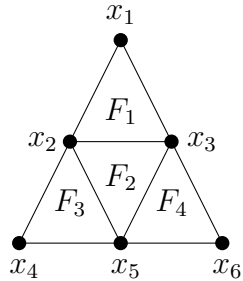 

\begin{lemma}
\label{AAAAA}
Let $\Delta$ be a quasi-forest on $\{x_1, \ldots, x_n\}$ and suppose that $F_1, \ldots, F_s$ is a leaf order.  Then there is a sequence $(c_1, \ldots, c_n)$ of positive integers and an integer $c$ for which $\sum_{x_i \in F_j}c_i = c$ for all $F_j$.    
\end{lemma}

\begin{proof}
Let $x_t, x_{t+1}, \ldots, x_n$ be the free vertices of $\Delta$ in $F_s$. Let $G=F_k$ be a branch of the leaf $F_s \in \Delta$. The simplicial complex $\langle F_1, \ldots, F_{s-1}\rangle$ on $\{x_1, \ldots, x_{t-1}\}$
is a quasi-forest with a leaf order $F_1, \ldots, F_{s-1}$.  Thus, by using induction on $s$, there
is a sequence $(c_1, \ldots, c_{t-1})$ of positive integers and an integer $c$ for which $\sum_{x_i \in F_j}c_i = c$ for all $F_j$ with $1 \leq j < s$.  (Clear, if $s=1$.)  Let $\sum_{x_i \in (G\cap F_{s})}c_i = c'$.  One may assume that $c - c' > n - t + 1$, since we can replace $(c_1,\ldots,c_{t-1})$ and $c$ with
$(dc_1,\ldots,dc_{t-1})$ and $dc$ for any positive
integer $d \geq 1$. Let $c_t, \ldots, c_n$ be positive integers with $c_t + \cdots + c_n = c - c'$. Note that $F_s\setminus G = \{x_t,x_{t+1},\ldots,x_n\}$. It then follows that $(c_1, \ldots, c_n)$ has the required property.
\end{proof}

Recall from \cite[Theorem 9.2.12]{HHgtm260} that the clique complex of a finite graph $G$ is a quasi-forest 
if and only if $G$ is a chordal graph.  A co-chordal
graph is a finite graph for which the complement 
graph $G^c$ of $G$ is chordal.  

\begin{theorem}\label{thm.cochordal}
A co-chordal graph is levelable.   
\end{theorem}

\begin{proof}
Let $G$ be a co-chordal graph on $\{x_1, \ldots, x_n\}$.  Then the set of maximal independent sets of $G$ is the set of maximal cliques of the chordal graph $G^c$.  Thus the desired result follows from Lemma \ref{AAAAA}.
 \end{proof}

\begin{example}
The graph $G$ of Figure \ref{fig:leaf} is levelable, since $G$ is co-chordal.  It follows from Theorem \ref{theorem:CW} that the path of $P_5$ of Figure \ref{fig:treeA} is non-levelable, thus
giving a proof different from Example \ref{ex.P5}. On the other hand, by Theorem \ref{Sapporo}, the tree of Figure \ref{fig:treeB} is levelable.
\end{example}
\begin{figure}[h]
\centering
\begin{tikzpicture}[scale=1.35]
\coordinate (x1) at (0,0){};
\coordinate (x2) at (1,0){};
\coordinate (x3) at (2,0){};
\coordinate (x4) at (3,0){};
\coordinate (x5) at (4,0){};
\fill(x1)circle(0.7mm);
\fill(x2)circle(0.7mm);
\fill(x3)circle(0.7mm);
\fill(x4)circle(0.7mm);
\fill(x5)circle(0.7mm);
\draw(x1)--(x2)--(x3)--(x4)--(x5);
\draw(x1)node[below=1mm]{{  $x_1$}};
\draw(x2)node[below=1mm]{{  $x_2$}};
\draw(x3)node[below=1mm]{{  $x_3$}};
\draw(x4)node[below=1mm]{{  $x_4$}};
\draw(x5)node[below=1mm]{{  $x_5$}};
\end{tikzpicture}
\caption{A non-levelable tree} \label{fig:treeA}
\end{figure}
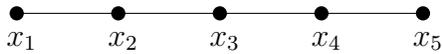 
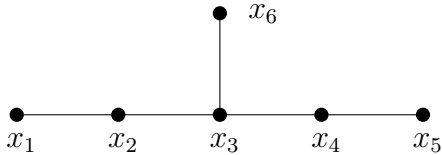
\begin{figure}[h]
\centering
\begin{tikzpicture}[scale=1.35]
\coordinate (x1) at (0,0){};
\coordinate (x2) at (1,0){};
\coordinate (x3) at (2,0){};
\coordinate (x4) at (3,0){};
\coordinate (x5) at (4,0){};
\coordinate (x6) at (2,1){};
\fill(x1)circle(0.7mm);
\fill(x2)circle(0.7mm);
\fill(x3)circle(0.7mm);
\fill(x4)circle(0.7mm);
\fill(x5)circle(0.7mm);
\fill(x6)circle(0.7mm);
\draw(x1)--(x2)--(x3)--(x4)--(x5);
\draw(x3)--(x6);
\draw(x1)node[below=1mm]{{  $x_1$}};
\draw(x2)node[below=1mm]{{  $x_2$}};
\draw(x3)node[below=1mm]{{  $x_3$}};
\draw(x4)node[below=1mm]{{  $x_4$}};
\draw(x5)node[below=1mm]{{  $x_5$}};
\draw(x6)node[right=1mm]{{  $x_6$}};
\end{tikzpicture}
\caption{A levelable tree} \label{fig:treeB}
\end{figure}

As mentioned earlier, 
it would be of interest to classify all
levelable chordal graphs. 
As a first step in achieving this 
goal,  one must classify all levelable trees. The
rest of this section works out these details.

Recall that a {\it free vertex} of a tree is a vertex which belongs to exactly one edge.  Every tree has at least two free vertices.  The {\it distance} of vertices $x$ and $y$ in a tree is the length of the unique path
between $x$ and $y$ in the tree.  Note that in a tree a path between $x$ and $y$ is unique. For a graph $G$ with subgraph $H$, we write $G\setminus H$ for the induced subgraph of $G$ whose vertex set is $V(G) \setminus V(H)$. For a vertex $x \in V(G)$, we write $G\setminus \{x\}$ for the induced subgraph whose vertex set is  $V(G) \setminus \{x\}$.

 \begin{theorem}\label{thm.trees}
 A tree is levelable if and only if each non-free vertex is adjacent to at least one free vertex.  
\end{theorem}

\begin{proof}
{\bf (``If'')} The ``If'' directions follows from
Theorem \ref{Sapporo}.  

{\bf (``Only If'')}
Suppose that a tree $G$ is levelable and that there is a non-free vertex $x$ that is adjacent to no free vertex. Let $\{x_1, \ldots, x_n\}$ be the vertex set of $G$ and assume that $(c_1, \ldots, c_n)$ is a weight function on $G$ with independence weight $f$. Let $G_1, \ldots, G_q$ with $q \geq 2$ denote the connected components of $G \setminus \{x\}$.  Each $G_i$ has at least one edge. 

Suppose that $q> 2$. We claim that
the graph $G \setminus G_q$ is levelable. 
To see this, fix a maximal
independent set $W$ of $G_q$ that contains no vertex
of $G_q$ that is adjacent to $x$.  Consider any 
maximal independent set $W'$ of $G \setminus G_q$.  Then
$W' \cup W$ is a maximal independent set of $G$.  Since
$G$ is levelable and the set $W$ is fixed, the equations
\[
\sum_{x_i \in W'}c_i = \sum_{x_i \in W' \cup W}c_i  - \sum_{x_i \in W}c_i=
f  - \sum_{x_i \in W}c_i
\]
show that 
$G \setminus G_q$ is levelable. Because
$q > 2$, $x$ cannot be a free vertex of $G \setminus G_q$,
and furthermore, $x$ is not adjacent to a free vertex, so
if $q-1 > 2$, we may repeat the above argument.
In particular, $G' = G \setminus G_q$ is a levelable
graph such that $G'\setminus \{x\}$ has $q-1$ connected 
components.  By repeating this process, we can
reduce to the case that $G$ is a levelable graph
with a non-free vertex $x$ that is not
adjacent to a free vertex, and $G\setminus \{x\} = G_1 \cup G_2$
has exactly two connected components.


So suppose the number of connected components is $q =2$.  Let  $A$ denote the set of vertices $y$ of $G$ for which the distance of $y$ and $x$ in $G$ is $2$.  Thus both $A\cap G_1$ and $A \cap G_2$ are nonempty.  Let $y_i$ denote the vertex of $G_i$ which is adjacent to $x$ for $i=1,2$.  Let $W$ be a maximal independent set of $G_1$ which contains $A\cap G_1$ and $W'$ a maximal independent set of $G_2$ which contains $A\cap G_2$.  Let $U$ be a maximal independent set of $G_1$ which contains $y_1$ and $U'$ a maximal independent set of $G_2$ which contains $y_2$.  Then 
\[
W\cup W' \cup \{x\}, \, \, \,  U \cup U', \, \, \, W' \cup U, \, \, \, W \cup U'
\]
are maximal independent sets of $G$.  Let $a = \sum_{x_i \in W} c_i$, $b = \sum_{x_i \in W'} c_i$, $c = \sum_{x_i \in U} c_i$ and $d = \sum_{x_i \in U'} c_i$. Let $x = x_1$.  It then follows 
that the positive integers 
\[c_1 + a + b,  \, \, \, c + d,  \, \, \, b + c,  \, \, \, a + d\]
are equal.  Hence $a =c$.  Thus $c_1=0$, a contradiction.
\end{proof}

\begin{remark}
    Levit and Tankus \cite{LT2015} described a polynomial time
    algorithm to determine if a graph that is $C_4, C_5$ and $C_6$-free
    is a weighted well-covered graph.  Since trees
    do not have any cycles, it can be shown
    that the algorithm of \cite{LT2015} produces the previous result,
    namely, that a tree is levelable if and only if every 
    non-free vertex is adjacent to at least one free vertex.
    \end{remark}

\begin{definition}
A graph $G$ is called a {\it caterpillar graph} if $G$ is a tree and if
    one removes all the vertices of degree one, the
    resulting graph is a connected path graph (called
    the central path).
    Each  vertex 
    that is 
    not on the central path is 
    called a
    {\it leg} of $G$. 
\end{definition}

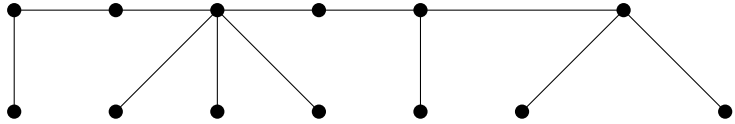
\begin{figure}[h]
\centering
\begin{tikzpicture}[scale=1.35]
\coordinate (a) at (0,1){};
\coordinate (b) at (1,1){};
\coordinate (c) at (2,1){};
\coordinate (d) at (3,1){};
\coordinate (e) at (4,1){};
\coordinate (f) at (6,1){};
\coordinate (g) at (0,0){};
\coordinate (h) at (1,0){};
\coordinate (i) at (2,0){};
\coordinate (j) at (3,0){};
\coordinate (k) at (4,0){};
\coordinate (l) at (5,0){};
\coordinate (m) at (7,0){};
\fill(a)circle(0.7mm);
\fill(b)circle(0.7mm);
\fill(c)circle(0.7mm);
\fill(d)circle(0.7mm);
\fill(e)circle(0.7mm);
\fill(f)circle(0.7mm);
\fill(g)circle(0.7mm);
\fill(h)circle(0.7mm);
\fill(i)circle(0.7mm);
\fill(j)circle(0.7mm);
\fill(k)circle(0.7mm);
\fill(l)circle(0.7mm);
\fill(m)circle(0.7mm);
\draw (g)--(a)--(b)--(c)--(d)--(e)--(f)--(m);
\draw (h)--(c);
\draw (i)--(c);
\draw (j)--(c);
\draw (k)--(e);
\draw (l)--(f);
\end{tikzpicture}
\caption{A caterpillar graph} \label{caterpillar}
\end{figure}

One will assume that each end vertex of the central path has at least one leg.  If not, then it can be regarded as a leg of 
the shorter
central path.

\begin{corollary}
\label{cat}
Let $G$ be a caterpillar graph.  Then $G$ is levelable if and only if every vertex of the central path has at least one leg.    
\end{corollary}

Theorem \ref{thm.trees} gives an alternative proof to 
Corollary \ref{cor.paths}.

\begin{corollary}
\label{PATH}
The path $P_n$ is levelable if and only 
if $n \in \{2,3,4\}$.
\end{corollary}

\begin{definition}
Let $q \geq 3$ be an integer and $n_1, \ldots, n_q$ positive integers.  Let $P_i$ denote the path of length $n_i$ on the vertex set $V_i=\{x, x_1^{(i)}, \ldots, x_{n_i}^{(i)}\}$ for $1 \leq i \leq q$.  We assume $V_i \cap V_j = \{x\}$ for $i \neq j$.  Let $G(n_1, \ldots, n_q) = P_1 \cup \cdots \cup P_q$.  We call $G(n_1, \ldots, n_q)$ a {\it big star graph} with a center $x$. 
\end{definition}

\begin{figure}[h]
\centering
\begin{tikzpicture}[scale=1.35]
\coordinate (a) at (0,0){};
\coordinate (b) at (.8,0){};
\coordinate (c) at (1.6,0){};
\coordinate (d) at (2.4,0){};
\coordinate (e) at (-.8,0){};
\coordinate (f) at (-1.6,0){};
\coordinate (g) at (0,.8){};
\coordinate (h) at (0,-.8){};
\coordinate (i) at (0,-1.6){};
\fill(a)circle(0.7mm);
\fill(b)circle(0.7mm);
\fill(c)circle(0.7mm);
\fill(d)circle(0.7mm);
\fill(e)circle(0.7mm);
\fill(f)circle(0.7mm);
\fill(g)circle(0.7mm);
\fill(h)circle(0.7mm);
\fill(i)circle(0.7mm);
\draw (f)--(e)--(a)--(b)--(c)--(d);
\draw (g)--(a)--(h)--(i);
\end{tikzpicture}
\caption{The big star graph $G(1,2,2,3)$} \label{bigstar}
\end{figure}
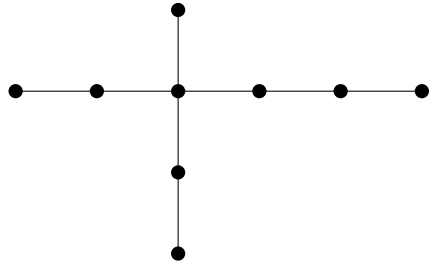

\begin{corollary}
\label{star}
A big star graph $G(n_1, \ldots, n_q)$ is levelable if and only if each $n_i \leq 2$ and there is at least 
one $j$ with $n_j = 1$.  
\end{corollary}


\section{Circulant graphs}

In this section, we concern ourselves with classifying circulant levelable graphs. 

\begin{definition}
    Let $n\geq 2$ and $S \subseteq \{1,\dots, \lfloor \frac{n}{2} \rfloor\}$. The \textit{circulant graph} $C_n(S)$ is the graph on the vertex set $\{x_1,\dots, x_n\}$ such that $\{x_i,x_j\} \in E(C_n(S))$ precisely when $|i-j| \in S$ or $n-|i-j| \in S$.
\end{definition}

Figure \ref{circulantex} shows an example of the circulant graph $C_{10}(2,5)$.  Theorem \ref{cor:multipartite} and
Corollary \ref{cor.complementcycle} already
classify when the circulant 
graphs $C_n(1,2,\ldots,\lfloor \frac{n}{2}\rfloor)$
and $C_n(2,\ldots,\lfloor \frac{n}{2} \rfloor)$
are levelable since 
$C_n(1,2,\ldots,\lfloor \frac{n}{2}\rfloor) \cong K_n$
and $C_n(2,\ldots,\lfloor \frac{n}{2} \rfloor) \cong C_n^c$.  We
now classify two further families of circulant graphs.

We first consider the case when $S=\{1\}$, which corresponds to the \textit{cycle graph} $C_n$ with vertices $V(C_n)=\{x_1,\dots,x_n\}$ and edges $E(C_n)= \{\{x_1,x_2\},\dots, \{x_{n-1},x_n\}, \{x_n,x_1\}\}$.

\begin{figure}[h]
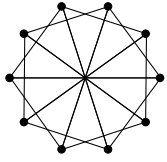

\Circulant{10}{2,5} 
\caption{The circulant $C_{10}(2,5)$}\label{circulantex}
\end{figure}

\begin{theorem}
    Let $G = C_n=C_n(1)$. Then $G$ is levelable if and only if $n \in \{2,3,4,5,7\}$. 
\end{theorem}

\begin{proof}
One can easily check that $C_n$ is well-covered for $n\in \{2,3,4,5,7\}$ and hence levelable. For $n=6$, the maximal independent sets of $C_6$ are $\{1,3,5\}, \{2,4,6\}, \{1,4\}, \{2,5\}$, and $\{3,6\}$. Taking $F_1=\{1,3,5\}$, $F_2=\{2,4,6\}$, $F_3=\{1,4\}$, and $F_4=\{2,5\}$, we see that the hypotheses of Lemma \ref{lem.noncondition}
are satisfied and so $C_6$ is not levelable.

We now show that $G=C_n$ is not levelable for $n\geq 8$.
If $n\geq 8$ is even, we define
    \begin{align*}
        F_1 &= \{ x_1,x_3,\dots,x_{n-1} \}, 
        ~~F_2 = \{ x_2,x_4,\dots,x_{n} \}, \\
        F_3 &= \{ x_1,x_3,\dots,x_{n-5},x_{n-2}\}, 
        ~~\mbox{and}
        ~~F_4 = \{x_2,x_4,\dots,x_{n-4},x_{n-1}\},
    \end{align*}
    and if $n\geq 9$ is odd, we define
    \begin{align*}
        F_1 &= \{ x_1,x_3,\dots,x_{n-2} \}, ~~
        F_2 = \{ x_2,x_4,\dots,x_{n-1} \}, \\
        F_3 &= \{ x_1,x_3,\dots,x_{n-8},x_{n-5},x_{n-2}\},
        ~~\mbox{and}~~
        F_4 = \{x_2,x_4,\dots,x_{n-7},x_{n-4},x_{n-1}\}.
    \end{align*}
    In either case, one can check that the hypotheses of Lemma \ref{lem.noncondition} are satisfied, and so it follows that $G$ is not levelable.
\end{proof}

A \textit{cubic circulant graph} is a circulant graph such that every vertex has degree 3. 
A cubic circulant graph must have the
form $C_{2n}(a,n)$ for some integers $1 \leq a < n$.
We seek to classify those cubic circulants that are levelable. We begin with a structural result for cubic circulants due to Davis and Domke.

\begin{theorem}[\cite{DDG}]\label{isomorphic-a-n}
Let $1 \leq a < n$  and $t={\rm gcd}(2n,a)$. 
\begin{enumerate}
\item[$(a)$] If $\frac{2n}{t}$ is even, then $C_{2n}(a,n)$ is isomorphic to $t$ copies of $C_{\frac{2n}{t}}(1,\frac{n}{t})$. 
\item[$(b)$] If $\frac{2n}{t}$ is odd, then $C_{2n}(a,n)$ is isomorphic to $\frac{t}{2}$ copies of $C_{\frac{4n}{t}}(2,\frac{2n}{t})$. 
\end{enumerate}
\end{theorem}

By Lemma \ref{lemma.connected}, it now suffices to determine for which $n \geq 2$ the graphs $C_{2n}(1,n)$ and $C_{2n}(2,n)$ are levelable. (Note that we may assume $n$ is odd for the latter.) Explicitly, these graphs have the vertex sets $V(C_{2n}(1,n))=\{x_1,\dots,x_{2n}\}$ and $V(C_{2n}(2,n))=\{x_1,\dots,x_{n},y_1\dots,y_n\}$ and edge sets
\[
E(C_{2n}(1,n))=\{\{x_i,x_{i+1}\}\mid 1\leq i \leq 2n-1\} \cup \{\{x_1,x_{2n}\}\} \cup \{\{x_i,x_{n+i}\} \mid 1\leq i\leq n\}
\]
and
\begin{align*}
E(C_{2n}(2,n)) =~&\{\{x_i,x_{i+1}\}\mid 1\leq i \leq n-1\} \cup \{\{y_i,y_{i+1}\}\mid 1\leq i \leq n-1\} \\
& \cup \{\{x_1,x_{n}\},\{y_1,y_{n}\}\} \cup \{\{x_i,y_{i}\} \mid 1\leq i\leq n\}
\end{align*}
respectively. Figure \ref{cubic-picture} shows the labeling of these graphs will use in our argument.

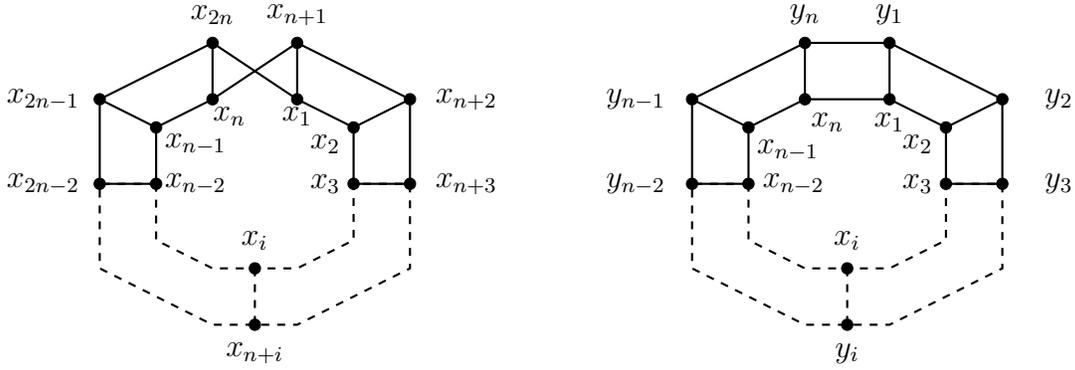
\begin{figure}[h]
\centering
\begin{tikzpicture}
[scale=.75]

\draw[thick] (-7,0) --(-7,1)--(-6,1.5)--(-4.5,2.5)--(-2.5,1.5)--(-2.5,0)--(-3.5,0)--(-3.5,1)--(-4.5,1.5)--(-6,2.5)--(-8,1.5)--(-8,0)--(-7,0);
\draw[thick] (-7,1)--(-8,1.5);
\draw[thick] (-6,1.5)--(-6,2.5);
\draw[thick] (-4.5,1.5)--(-4.5,2.5);
\draw[thick] (-3.5,1)--(-2.5,1.5);

\draw[thick,dashed] (-7,0) --(-7,-1)--(-6,-1.5)--(-4.5,-1.5)--(-3.5,-1)--(-3.5,0)--(-2.5,0)--(-2.5,-1.5)--(-4.5,-2.5)--(-6,-2.5)--(-8,-1.5)--(-8,0)--(-7,0);
\draw[thick,dashed] (-5.25,-1.5) --(-5.25,-2.5);

\draw [fill] (-7,0)  circle [radius=0.1];
\draw [fill] (-7,1) circle [radius=0.1];
\draw [fill] (-6,1.5) circle [radius=0.1];
\draw [fill] (-4.5,2.5) circle [radius=0.1];
\draw [fill] (-2.5,1.5) circle [radius=0.1];
\draw [fill] (-2.5,0) circle [radius=0.1];
\draw [fill] (-3.5,0) circle [radius=0.1];
\draw [fill] (-3.5,1) circle [radius=0.1];
\draw [fill] (-4.5,1.5) circle [radius=0.1];
\draw [fill] (-6,2.5)circle [radius=0.1];
\draw [fill] (-8,1.5) circle [radius=0.1];
\draw [fill] (-8,0) circle [radius=0.1];

\draw [fill] (-5.25,-1.5) circle [radius=0.1];
\draw [fill] (-5.25,-2.5) circle [radius=0.1];

\node at (-6.3,0) {$x_{n-2}$};
\node at (-6.3,0.7) {$x_{n-1}$};
\node at (-5.7,1.2) {$x_{n}$};
\node at (-4.5,3) {$x_{n+1}$};
\node at (-1.5,1.5) {$x_{n+2}$};
\node at (-1.5,0) {$x_{n+3}$};
\node at (-4,0) {$x_{3}$};
\node at (-4,0.7) {$x_{2}$};
\node at (-4.5,1.2) {$x_{1}$};
\node at (-6,3) {$x_{2n}$};
\node at (-9,1.5) {$x_{2n-1}$};
\node at (-9,0) {$x_{2n-2}$};

\node at (-5.25,-1) {$x_{i}$};
\node at (-5.25,-3) {$x_{n+i}$};

\draw[thick] (3.5,0) --(3.5,1)--(4.5,1.5)--(6,1.5)--(7,1)--(7,0)--(8,0)--(8,1.5)--(6,2.5)--(4.5,2.5)--(2.5,1.5)--(2.5,0)--(3.5,0);
\draw[thick] (3.5,1)--(2.5,1.5);
\draw[thick] (4.5,1.5)--(4.5,2.5);
\draw[thick] (6,1.5)--(6,2.5);
\draw[thick] (7,1)--(8,1.5);

\draw[thick,dashed] (3.5,0) --(3.5,-1)--(4.5,-1.5)--(6,-1.5)--(7,-1)--(7,0)--(8,0)--(8,-1.5)--(6,-2.5)--(4.5,-2.5)--(2.5,-1.5)--(2.5,0)--(3.5,0);
\draw[thick,dashed] (5.25,-1.5) --(5.25,-2.5);

\draw [fill] (3.5,0)  circle [radius=0.1];
\draw [fill] (3.5,1) circle [radius=0.1];
\draw [fill] (4.5,1.5) circle [radius=0.1];
\draw [fill] (6,1.5) circle [radius=0.1];
\draw [fill] (7,1) circle [radius=0.1];
\draw [fill] (7,0) circle [radius=0.1];
\draw [fill] (8,0) circle [radius=0.1];
\draw [fill] (8,1.5) circle [radius=0.1];
\draw [fill] (6,2.5) circle [radius=0.1];
\draw [fill] (4.5,2.5)circle [radius=0.1];
\draw [fill] (2.5,1.5) circle [radius=0.1];
\draw [fill] (2.5,0) circle [radius=0.1];

\draw [fill] (5.25,-1.5) circle [radius=0.1];
\draw [fill] (5.25,-2.5) circle [radius=0.1];

\node at (4.3,0) {$x_{n-2}$};
\node at (4.2,0.6) {$x_{n-1}$};
\node at (4.9,1.1) {$x_{n}$};
\node at (6,1.1) {$x_{1}$};
\node at (6.5,0.7) {$x_{2}$};
\node at (6.5,0) {$x_{3}$};
\node at (9,0) {$y_{3}$};
\node at (9,1.5) {$y_{2}$};
\node at (6,3) {$y_{1}$};
\node at (4.5,3) {$y_{n}$};
\node at (1.5,1.5) {$y_{n-1}$};
\node at (1.5,0) {$y_{n-2}$};

\node at (5.25,-1) {$x_{i}$};
\node at (5.25,-3) {$y_{i}$};

\end{tikzpicture}
\caption{The graphs $C_{2n}(1,n)$ and $C_{2n}(2,n)$ (with $n$ odd).}\label{cubic-picture}
\end{figure}

\begin{theorem}\label{thm.circulant1}
Let $G=C_{2n}(1,n)$. Then $G$ is levelable if and only if $n \in \{2,3,4\}$.
\end{theorem}

\begin{proof}
   For $n \in \{2,3,4\}$, one can directly check that the graphs $C_{2n}(1,n)$ are well-covered and hence levelable (also 
   see \cite[Theorem 4.3]{BH2011}). So suppose $n\geq 5$. We will show that $G=C_{2n}(1,n)$ is not levelable. We consider two cases.

    Case 1: If $n\geq 5$ is odd, define
    \begin{align*}
        F_1 &= \{ x_1,x_3,\dots,x_{2n-1} \}, \\
        F_2 &= \{ x_2,x_4,\dots,x_{2n} \}, \\
        F_3 &= \{ x_1,x_3,\dots,x_{n-4},x_{n-1}\} \cup \{x_{n+2},x_{n+4} ,\dots, x_{2n-3}\}, \\
        F_4 &= \{x_2,x_4,\dots,x_{n-3},x_{n}\} \cup \{x_{n+3},x_{n+5} ,\dots, x_{2n-2}\}.
    \end{align*}
    We claim that $F_1,\dots,F_4$ are maximal independent sets of $G$ that satisfy the hypotheses of Lemma \ref{lem.noncondition}.

    Since $F_3 \cup F_4 \subsetneq F_1 \cup F_2$ and $F_3 \cap F_4= \emptyset$, it only remains to check that $F_1,\dots,F_4$ are maximal independent sets of $G$. We see that $F_1,\dots, F_4$ are independent sets by construction.
    Since $\{x_1,x_{2n}\}$ and $\{x_i,x_{i+1}\}$ are edges of $G$ for each $1 \leq i \leq 2n-1$, it is clear that $F_1$ and $F_2$ are maximal. Additionally, since $\{x_{n-1},x_{2n-1}\}$ (resp. $\{x_{n},x_{2n}\}$) is an edge of $G$, we conclude that $F_3$ (resp. $F_4$) is maximal. It follows from Lemma \ref{lem.noncondition} that $G$ is not levelable.

    Case 2: If $n \geq 6$ is even, define
    \begin{align*}
        F_1 &= \{ x_1,x_3,\dots,x_{n-1}, x_{n+2}, x_{n+4},\dots, x_{2n-2} \}, \\
        F_2 &= \{ x_2,x_4,\dots,x_{n}, x_{n+3}, x_{n+5},\dots, x_{2n-1}\}, \\
        F_3 &= \{x_1,x_3,\dots,x_{n-5},x_{n-1}, x_{2n-3}\} \cup A, \\
        F_4 &= \{x_2,x_4,\dots,x_{n-4},x_{n},x_{2n-2}\} \cup A',
    \end{align*}
    where $A=A'=\emptyset$ if $n=6$ and \begin{align*}
        A &=\{x_{n+2},x_{n+4},\dots,x_{2n-6}\}, \\
        A' &=\{x_{n+3},x_{n+5},\dots,x_{2n-5}\}
        \end{align*}
    for all $n\geq 8$.
   
    We proceed as above. Since $F_3 \cup F_4 \subsetneq F_1 \cup F_2$ and $F_3 \cap F_4 = \emptyset$, it only remains to check that $F_1,\dots,F_4$ are maximal independent sets of $G$. We see that $F_1,\dots, F_4$ are independent sets by construction. Since $\{x_1,x_{2n}\}$ and $\{x_i,x_{i+1}\}$ are edges of $G$ for each $1 \leq i \leq 2n-1$, it is clear that $F_1$ and $F_2$ are maximal. Additionally, since $\{x_{n-3},x_{2n-3}\}$ and $\{x_{n-1},x_{2n-1}\}$ (resp. $\{x_{n-2},x_{2n-2}\}$ and $\{x_{n},x_{2n}\}$) are edges of $G$, we conclude that $F_3$ (resp. $F_4$) is maximal. It follows from Lemma \ref{lem.noncondition} that $G$ is not levelable.
\end{proof}

\begin{theorem} \label{thm.circulant2}
Let $G=C_{2n}(2,n)$ with $n$ odd. Then $G$ is levelable if and only if $n \in \{3,5\}$.
\end{theorem}

\begin{proof}
For $n \in \{3,5\}$, one can directly check that the graphs $C_{2n}(2,n)$ are well-covered 
(also 
   see \cite[Theorem 4.3]{BH2011}) and hence levelable. So suppose $n\geq 7$ is odd. Define
    \begin{align*}
        F_1 &= \{x_1,x_3\dots,x_{n-2}\}  \cup \{y_2,y_4\dots, y_{n-1} \},  \\
        F_2 &= \{x_2,x_4\dots,x_{n-1}\}  \cup \{y_3,y_5\dots, y_{n} \}, \\
        F_3 &= \{ x_1,x_3,\dots,x_{n-6}, x_{n-3} \} \cup \{ y_2,y_4,\dots,y_{n-5}, y_{n-1} \}, \\
        F_4 &= \{x_2,x_4,\dots,x_{n-5}, x_{n-2} \} \cup \{ y_3,y_5,\dots,y_{n-4}, y_{n} \}.
    \end{align*}
    Since $F_3 \cup F_4 \subsetneq F_1 \cup F_2$ and $F_3 \cap F_4  =\emptyset$, it only remains to check that $F_1,\dots,F_4$ are maximal independent sets of $G$. We see that $F_1,\dots, F_4$ are independent sets by construction. Since $\{x_1,x_{n}\}$, $\{y_1,y_{n}\}$, $\{x_i,x_{i+1}\}$, and $\{y_i,y_{i+1}\}$  are edges of $G$ for each $1 \leq i \leq n-1$, it is clear that $F_1$ and $F_2$ are maximal. Additionally, since $\{x_{n-1},y_{n-1}\}$ and $\{x_{n-3},y_{n-3}\}$ (resp. $\{x_{n},y_{n}\}$ and $\{x_{n-2},y_{n-2}\}$) are edges of $G$, we conclude that $F_3$ (resp. $F_4$) is maximal. It follows from Lemma \ref{lem.noncondition} that $G$ is not levelable.
\end{proof}


\section{Levelable graphs and a connection to level artinian rings}
This last section describes the origins of levelable graphs and
their connection to level artinian rings.  Level rings, which 
were introduced by Stanley \cite{Stan77} almost 50 years ago,
have inspired a significant amount of research in 
combinatorial commutative algebra.
As we shall show, our new connection allows us to use
the properties of levelable graphs to deduce consequences
for the existence of level  artinian  rings of a certain form.

Let $A$ be a finitely generated graded $\mathbb{K}$-algebra, i.e., 
$A = \mathbb{K}[x_1,\ldots,x_n]/I$ for some homogeneous ideal
$I$. We will define 
level rings only in the case that $A$ is artinian -- 
the  general
definition of level rings is in terms of the canonical module of $A$.
The ring $A$ is {\it artinian} if the Krull dimension of $A$ is $0$.
Being artinian is equivalent to the fact that there exists positive
integers $a_1,\ldots,a_n$ such that
$x_i^{a_i} \in I$ for $i=1,\ldots,n$.  
If $A$ is a graded artinian ring, then there is an integer $e$ such that
\[A = A_0 \oplus A_1 \oplus A_2 \oplus \cdots \oplus A_e ~~\mbox{with
$A_e \neq 0$}\]
where each $A_i$ is a $\mathbb{K}$-vector space of homogeneous elements of 
$A$ of degree $i$.
Note that $A_i = 0$ for all $i > e$.   
The {\it socle} of $A$ is the homogeneous ideal of $A$ defined by
\[{\rm  soc}(A) = \{a \in A ~|~ a\overline{x}_i = 0 ~~\mbox{for $i=1,\ldots,n$}\},\]
i.e., ${\rm soc}(A)$ is 
the annihilator of the maximal ideal
$\langle \overline{x}_1,\ldots,\overline{x}_n \rangle$ in $A$.
The {\it socle-vector} is $s(A) = (s_0,s_1,\ldots,s_e)$ where
$s_i = \dim_{\mathbb{K}} {\rm soc}(A)_i$.  We 
come to the key definition:
\begin{definition}
A graded artinian algebra $A$ is {\it level} if the socle-vector
of $A$ has the form $s(A) = (0,0,\ldots,0,s_e)$ for some integer $s_e > 0$.  
\end{definition}

\begin{example}
    Suppose that $A = \mathbb{K}[x,y]/I$ where $I = 
    \langle x^3, y^3 \rangle$.  Any monomial $m= x^ay^b$ 
    of degree five or larger is divisible by either 
    $x^3$ or $y^3$.
    So $A$ is an artinian ring of the form
    $A  = A_0 \oplus A_1 \oplus A_2 \oplus A_3 \oplus A_4$
    where
    $$
    \begin{tabular}{lll}
    $A_0 = {\rm span}_{\mathbb{K}}\{\overline{1}\}$ &
    $A_1 = {\rm span}_{\mathbb{K}}\{\overline{x},\overline{y}\}$ &
    $A_2 = {\rm span}_{\mathbb{K}}\{\overline{x^2},
    \overline{xy},\overline{y^2}\}$  \\
    $A_3 = {\rm span}_{\mathbb{K}}\{\overline{x^2y},\overline{xy^2}\}$ &
    $A_4 = {\rm span}_{\mathbb{K}}\{\overline{x^2y^2}\}.$
 &    \end{tabular}
    $$
    where $\overline{m}$
    denotes the equivalence class of $m$ in $A$. 
    For each of the basis elements $\overline{m}$, we 
    can check if $x\overline{m} =0$ and $y\overline{m} =0$.
    It is clear that $x\overline{x^2y^2} = \overline{x^3y^2} =0$
    and $y\overline{x^2y^2} = \overline{x^2y^3} =0$. On
 the other hand, no other basis element is annihilated by
    both $x$ and $y$. Hence
    $${\rm soc}(A) = A_4 ~~\mbox{and}~~ s(A) = (0,0,0,0,1).$$
    Thus the graded artinian algebra $A$ is level.
\end{example}

In 2010, Van Tuyl and Zanello \cite{VTZ2010} introduced {\it levelable simplicial complexes}; these are simplicial complexes
that allow one to build a level graded artinian ring.  We
recall the main ideas of Stanley--Resiner theory -- see \cite{HHgtm260} for more details.
A simplicial complex on a set $V = \{x_1,\ldots,x_n\}$ is 
a subset of the power set of $V$ with the property that
if $F \in \Delta$ and $G \subseteq F$, then $G \in \Delta$.  
The elements of $\Delta$ are called {\it faces}, and the
maximal faces of $\Delta$ under inclusion are the {\it facets}.
If $F_1,\ldots,F_s$ is a complete list of the facets of
$\Delta$, then we write $\Delta = \langle F_1,\ldots,F_s \rangle$.
The Stanley--Reisner correspondence allows one to associate
a simplicial complex $\Delta$ on $V$ with a squarefree
monomial ideal in $R =\mathbb{K}[x_1,\ldots,x_n]$ with $\mathbb{K}$
as field.  Precisely, 
\[I_\Delta = \langle x_{i_1}\cdots x_{i_s} ~:~ \{x_{i_1},\ldots,
x_{i_s}\} \not\in \Delta \rangle.\]   
The ideal $I_\Delta$ is the {\it Stanley--Reisner ideal}
of $\Delta$, and $R/I_\Delta$ is the {\it Stanley--Reisner ring}.

Given any simplicial complex and integers $a_1,\ldots,a_n \geq 2$,
set 
\[A(\Delta,(a_1,\ldots,a_n)) = \frac{\mathbb{K}[x_1,\ldots,x_n]}{I_\Delta + \langle
x_1^{a_1},\ldots,x_n^{a_n}\rangle}.\]
The ring
$A(\Delta,(a_1,\ldots,a_n))$ is an example of a 
graded artinian ring. We omit the case that
$a_i=1$ for some $i$ because in this situation
\[A(\Delta,(a_1,\ldots,a_n)) \cong A(\Delta',(a_1,\ldots,\hat{a}_i,
\ldots,a_n))\]
where $\Delta'$ is the induced simplicial complex on
$V \setminus \{x_i\}$.

In general, whether or not $A(\Delta,(a_1,\ldots,a_n))$ is
a level ring will depend upon both $\Delta$ and $(a_1,\ldots,a_n)$.
Van Tuyl and Zanello called $\Delta$ a {\it levelable simplicial
complex} if there exists  $(a_1,\ldots,a_n) \in \mathbb{N}^n$
with all $a_i \geq 2$ such that $A(\Delta,(a_1,\ldots,a_n))$
is a level ring.  As shown by Van Tuyl and Zanello,
determining if $\Delta$ is levelable reduces to solving
a particular system of equations constructed from the
facets of $\Delta$.

\begin{theorem}[{\cite[Theorem 6]{VTZ2010}}]\label{levelableclassification}
Let $\Delta$ be
a simplicial complex on $V = \{x_1,\ldots,x_n\}$ with facets $\{F_1,\ldots,F_s\}$.
Let $F_i = \{x_{i,1},\ldots,x_{i,d_i}\}$ denote the $i$-th facet.
Then $A(\Delta,(a_1,\ldots,a_n))$
is level if and only if $(a_1,\ldots,a_n)$ is an integral solution 
with $a_i \geq 2$ to the 
system
\begin{eqnarray*}
    (x_{1,1}+ x_{1,2} + \cdots + x_{1,d_1}) -
    (x_{2,1} + x_{2,2} + \cdots + x_{2,d_2}) &=& d_1 -d_2 \\
    (x_{2,1}+ x_{2,2} + \cdots + x_{2,d_2}) -
    (x_{3,1} + x_{3,2} + \cdots + x_{3,d_3}) &=& d_2 -d_3 \\
   & \vdots & \\
    (x_{s-1,1}+ x_{s-1,2} + \cdots + x_{s-1,d_{s-1}}) -
    (x_{s,1} + x_{s,2} + \cdots + x_{s,d_s}) &=& d_{s-1} -d_s. 
\end{eqnarray*}
\end{theorem}

Levelable graphs, as we have defined them in this paper,
arise when considering when the independence complex of a graph is 
a levelable simplicial complex.  The {\it independence complex}
of a graph $G$ is the simplicial complex
\[
{\rm Ind}(G) = \{W \subseteq V ~:~ \mbox{$W$ is an  independent set of $V$}\},
\]
that is, the faces of ${\rm Ind}(G)$ are 
the independent sets of $G$. 
The facets of ${\rm Ind}(G)$ is the set ${\rm MaxInd}(G)$ that we 
defined in the introduction.  The connection between levelable
graphs and level graded artinian rings is then captured in the following
result.

\begin{theorem}
    A graph $G$ is a levelable
    graph if and only if ${\rm Ind}(G)$ is a levelable
    simplicial complex.
\end{theorem}

\begin{proof}
Let ${\rm MaxInd}(G) = \{W_1,\ldots,W_s\}$ be the maximal
independent sets of $G$, and let $W_i = \{x_{i,1},\ldots,
x_{i,d_i}\}$ for $i=1,\ldots,s$.  Note that
$W_1,\ldots,W_s$ are also the facets of ${\rm Ind}(G)$.

{\bf (``Only If'')}  Suppose $G$ is a levelable graph.  Hence there exists a weight function $(c_1,\ldots,c_n)$ with positive integers 
and independence weight $c$ such that
\[
    c_{i,1} + \cdots + c_{i,d_i}  =  c  
    ~\mbox{for $i=1,\ldots,s$}. 
\]
If we now consider the vector $(c_1+1,\ldots,c_n+1)$, each
$W_i$ gives the equality:
\[(c_{i,1}+1)+\cdots +(c_{i,d_i}+1) = c+d_i.\]
But then for  $i=1,\ldots,s-1$, we have
\[((c_{i,1}+1)+\cdots+(c_{i,d_i}+1)) - 
((c_{i+1,1}+1)+\ldots + (c_{i+1, d_{i+1}}+1)) =
d_i-d_{i+1}.\]
By Theorem \ref{levelableclassification}, this
means that ${\rm Ind}(G)$ is a levelable simplicial 
complex since $(c_1+1,\ldots,c_n+1)$ is  solution
to the system of equations with $c_i+1 \geq 2$ for all $i$.

{\bf (``If'')}  Let $(a_1,\ldots,a_n)$ with all $a_i \geq 2$
be an integral solution to the equations of
Theorem \ref{levelableclassification} constructed
from the facets of ${\rm Ind}(G)$.  By reversing
the above argument, we can show $G$ is levelable
with respect to the weight function $(a_1-1,\ldots,a_n-1)$.
\end{proof}

In light of the previous result,
finding levelable graphs immediately 
allows us to construct a level graded artinian  ring.  The 
{\it edge ideal} of $G$ is the squarefree monomial ideal
$I(G) = \langle x_ix_j : \{x_i,x_j\} \in E(G) \rangle$.  It is 
well-known (e.g., see \cite{HHgtm260}) that $I(G)$ is the 
Stanley--Reisner ideal of the simplicial complex ${\rm Ind}(G)$, that is, $I(G) = I_{{\rm Ind}(G)}$.
By simply translating definitions and using the previous theorem,
we deduce Theorem \ref{maintheorem}:

\begin{theorem}\label{thm.constructartinina}
    Let $G = (V,E)$ be a graph with edge ideal $I(G)$ in 
    $R = \mathbb{K}[x_1,\ldots,x_n]$.  Then
    $G$ is a levelable graph with respect to the
    weight function $(c_1,\ldots,c_n)$ if and only if
    \[R/(I(G)+\langle x_1^{c_1+1},\ldots,x_n^{c_n+1}\rangle)\]
    is a level graded artinian ring.
\end{theorem}

As the previous result shows, constructing level graded artinian rings from
an edge ideal of a graph $G$ reduces to verifying that $G$ is a levelable graph.
We conclude this paper by using 
a result of Brown and Nowakowski \cite{BN2007}
to show that level graded artinian 
rings of this form are quite rare.

Recall from the introduction that we can associate to any
$G$ a vector space
\footnotesize
\[{\rm WCW}(G) = \{ {\bf c} = 
(c_1,\ldots,c_n) : \mbox{$G$ with weight function {\bf c} is a weighted well-covered graph}\} \subseteq F^{|V|}.\]
\normalsize
for a fixed field $F$.  
Note that to construct ${\rm WCW}(G)$ 
we must first fix a field.  If the field has characterisic zero,
then $F$ contains a copy of $\mathbb{N}$. We then have
the following fact:

\begin{lemma}\label{lem.wcwdim}
    Fix a field $F$ of characteristic zero. Let $G$ be a graph
    and view ${\rm WCW}(G)$ as a vector space over $F$. 
    If $\dim {\rm WCW}(G) = 0$, then $G$ is not levelable.
\end{lemma}

\begin{proof}
    If $G$ was levelable, then there is a weight vector $(c_1,\ldots,c_n)$
    with $c_i > 0$ and $c_i \in \mathbb{N} \subseteq F$.  But then
    $(c_1,\ldots,c_n) \in {\rm WCW}(G)$ is a non-zero vector,
    and thus $0 = \dim {\rm WCW}(G) \geq 1$, a contradiction.
\end{proof}
\begin{example}
    The converse of the previous statement is false.  As
    we have seen, $P_5$ is not a levelable graph. However
    $(1,1,0,-1,-1) \in {\rm WCW}(P_5)$, so $\dim {\rm WCW}(P_5) \neq 0$.
\end{example}

We recall some of the terminology related to random graphs.  A {\it random graph
$G$} of order $n$ is a graph on $V = \{x_1,\ldots,x_n\}$ where there is
an edge between $x_i$ and $x_j$ of fixed probability $p \in (0,1)$ for
all possible $1 \leq i  < j \leq n$.  We say that {\it almost all graphs $G$} 
have property $P$ if ${\rm Prob}(\mbox{$G$ has property $P$}) \rightarrow 1$
as $|V| \rightarrow \infty$ for any $0 < p < 1$.

\begin{theorem}[{\cite[Theorem 1; Corollary 2]{BN2007}}]\label{thm.almostall}  
Fix a field $F$. For almost all graphs $G$, $\dim {\rm WCW}(G) = 0$.
In particular, for almost all graphs $G$, $G$ is not well-covered.
\end{theorem}

We can now deduce the following results that are of interest to
those working in combinatorial commutative algebra.

\begin{corollary}\label{cor.neverlevel}
    For almost all graphs $G$, the graded artinian ring
    \[R/(I(G)+\langle x_1^{a_1},\ldots,x_n^{a_n}\rangle)\]
    is not level for any choice of $(a_1,\ldots,a_n) \in \mathbb{N}^n$
    with $a_i \geq 2$ for all $i$.
\end{corollary}

\begin{proof}
    By Theorem \ref{thm.almostall}, almost all graphs $G$
    have $\dim {\rm WCW}(G) = 0$, which by Lemma \ref{lem.wcwdim} 
    implies that almost all graphs are not levelable, and thus
    by Theorem \ref{thm.constructartinina}, for almost all graphs,
    the ring in the statement of the corollary cannot be level.
\end{proof}

One class of rings that are of great importance in commutative
algebra are Cohen--Macaulay rings.  We say a
ring $S$ is {\it Cohen--Macaulay} if
$\dim S = {\rm depth} ~S$.  A graph $G$ is a {\it Cohen--Macaulay
graph} if the ring $R/I(G)$ is a Cohen--Macaulay ring (for more details,
see \cite{HHgtm260}).
For $G$ to be Cohen--Macaulay, a necessary property is that $G$ is well-covered (see \cite[Lemma 9.1.10]{HHgtm260}).
Consequently, we can also deduce the following 
result from Theorem \ref{thm.almostall}:

\begin{corollary}\label{cor.nevercm}
    For almost all graphs $G$, $G$ is not Cohen--Macaulay.
\end{corollary}

\begin{remark}
    The previous result was also proved by
    Dochtermann--Newman \cite[Corollary 1.4]{DN2023} and 
    Erman--Yang \cite[Corollary 7.1]{EY2018}, using approaches different than Brown--Nowakowski.
\end{remark}

\section{Future directions}

We hope that our introduction
of levelable graphs will inspire further research, 
both in graph theory and commutative algebra.  We conclude
with some possible research questions inspired by 
this work.

One natural direction is to describe other 
families of levelable graphs.  More generally, it would
be interesting to answer the following question:

\begin{question}
    What are conditions on a graph $G$ that are either
    sufficient for $G$ to be levelable or provide an
    obstruction to $G$ being levelable?
\end{question}

\noindent
Lemma \ref{lem.noncondition} can be viewed as a partial
answer to this question in that it provides an obstruction
to a graph being levelable.  Finding other such conditions would be of interest, especially in light of 
Theorem \ref{thm.constructartinina} and the connection
to the construction of artinian rings.

Continuing with this connection to artinian rings,
it would be of interest to explore the properties
of these rings, including their Hilbert functions.
Recall that if $I$ is a homogeneous ideal in
a polynomial ring $R = \mathbb{K}[x_1,\ldots,x_n]$, then
the {\it Hilbert function} of $R/I$ is given
by $H_{R/I}(t) = \dim_\mathbb{K} R_t  - \dim_{\mathbb{K}} I_t$
where $R_t$, respectively $I_t$, is the vector
space of homogeneous forms of degree $t$ in $R$, respectively $I$.
We can then ask:

\begin{question}
    Suppose that $G$ is a levelable graph with
    respect to the weight function $(c_1,\ldots,c_n)$.
    What is the Hilbert function of the level
    ring $R/(I(G) + \langle x_1^{c_1+1},\ldots,x_n^{c_n+1}\rangle)$?  In particular,
    how is it related to the graph $G$ and the weight function?
\end{question}
\noindent
This question is connected to an ongoing research program in
commutative algebra that is interested in classifying the
Hilbert functions of level rings;  see
\cite{GHMS2007,HSZ,HLO,H88,H89,MNZ} for a
sample of this work.  

In another direction, it may be of interest to explore
the set of vectors that make a graph $G$ levelable.  
In particular, define 
$${\rm WL}(G) = \{(c_1,\ldots,c_n) \in \mathbb{N}_{>0}^n ~|~ \mbox{
$G$ is levelable with respect to $(c_1,\ldots,c_n)$}\}.$$
Then the set of ${\rm WL}(G)$ is a semigroup
of integer lattice points contained in 
the vector space ${\rm WCW}(G)$.
This leads to an interesting question:
\begin{question}
    Suppose that $G$ is a levelable graph. 
    Does the set of lattice points
    ${\rm WL}(G)$ have any interesting
    geometric or combinatorial information?  
\end{question}

\subsection*{Acknowledgments}

The authors thank the referees for their
careful reading and improvements to the paper.
Some of our results first appeared in the senior undergraduate thesis
of M. Chong \cite{Cthesis} under the supervision of A. Van Tuyl. 
We thank David Tankus for his comments, and
in particular, for his improvement to Corollary \ref{cor:AAA} and
bringing \cite{LT2015} to our attention.
Part of this work was carried out while T. Hibi visited
the other authors at McMaster University, and when T. Hibi and A. Van Tuyl
visited the Fields Institute in Toronto as part of the ``Thematic Program
in Commutative Algebra and Applications''.  We thank the hospitality of 
both institutions.
A. Van Tuyl’s research is supported by NSERC Discovery Grant 2024-05299.

\subsection*{Conflict of interest.}
On behalf of all authors, the corresponding author states that there is no conflict of interest. 

\subsection*{Data availability.} This document
has no associated data.

\end{document}